\documentclass{amsart}
\usepackage{amssymb}
\usepackage{booktabs}
\makeatletter

\let\Re\undefined

\DeclareMathOperator{\Re}{Re}

\DeclareMathOperator{\Ind}{Ind}

\DeclareMathOperator{\ch}{ch}
\DeclareMathOperator{\id}{id}
\DeclareMathOperator{\ord}{ord}
\DeclareMathOperator{\vol}{vol}
\DeclareMathOperator{\meas}{meas}
\DeclareMathOperator{\hmi}{hmi}

\newcommand{\dd}[1]{\mathop{\mathrm{d}#1}}

\newcommand{\al}{\alpha}
\newcommand{\be}{\beta}
\newcommand{\ga}{\gamma}
\newcommand{\de}{\delta}

\newcommand{\la}{\lambda}
\newcommand{\ze}{\zeta}
\newcommand{\wg}{\varpi}
\newcommand{\ph}{\varphi}
\newcommand{\om}{\omega}
\newcommand{\De}{\varDelta}
\newcommand{\Th}{\varTheta}
\renewcommand{\Pi}{\varPi}

\newcommand{\psm}[2][\Biggl(]
  {\left(\vphantom{#1}\begin{smallmatrix}#2\end{smallmatrix}\right)}

\newcommand\abs{\@ifstar{\abs@star}{\abs@nostar}}
\newcommand{\abs@nostar}[1]{\lvert#1\rvert}
\newcommand{\abs@star}[1]{\left\lvert#1\right\rvert}
\newcommand\Abs{\@ifstar{\Abs@star}{\Abs@nostar}}
\newcommand{\Abs@nostar}[1]{\lVert#1\rVert}
\newcommand{\Abs@star}[1]{\left\lVert#1\right\rVert}

\newcommand\set[1]{\Bigl\{#1\Bigr\}}

\newcommand{\cj}[1]{{\overline{#1}}}

\newcommand{\ade}{\mathbb{A}}

\newcommand{\lqu}{\backslash}

\newcommand{\makegroup}[7][k]{%
  \newcommand{#3}[1][#1]{{#2_{##1}}}%
  \newcommand{#4}{{#2_\ade}}%
  \newcommand{#5}[1][#1]{{#3[##1]\lqu#4}}%
  \newcommand{#6}[1][v]{{#2_{##1}}}
  \newcommand{#7}{{#2_\infty}}}

\makegroup{G}{\Gk}{\Ga}{\Gq}{\Gv}{\Gin}
\makegroup{H}{\Hk}{\Ha}{\Hq}{\Hv}{\Hin}
\makegroup{M}{\Mk}{\Ma}{\Mq}{\Mv}{\Min}
\makegroup{N}{\Nk}{\Na}{\Nq}{\Nv}{\Nin}
\makegroup{P}{\Pk}{\Pa}{\Pq}{\Pv}{\Pin}
\makegroup{Q}{\Qk}{\Qa}{\Qq}{\Qv}{\Qin}
\makegroup{\Th}{\Sk}{\Sa}{\Sq}{\Sv}{\Sin}

\newcommand{\GL}{\mathrm{GL}}
\newcommand{\gO}{\mathrm{O}}

\newcommand{\CC}{\mathbb{C}}
\newcommand{\RR}{\mathbb{R}}
\newcommand{\QQ}{\mathbb{Q}}

\newcommand{\kx}{{k^\times}}

\newcommand{\oo}{\mathfrak o}
\newcommand{\ox}{\mathfrak o^\times}

\newcommand{\Cx}{{\CC^\times}}

\renewcommand{\thmhead}[3]{\thmnumber{{\rm(#2)}\@ifnotempty{#1#3}%
    { }}\thmname{#1\@ifnotempty{#3}{\,---\,}}\thmnote{#3}}
\newcommand{\thref}{\eqref}

\newtheorem{prop}[equation]{Proposition}
\newtheorem{lem}[equation]{Lemma}
\newtheorem*{thm*}{Theorem}
\newtheorem*{prop*}{Proposition}
\newtheorem*{lem*}{Lemma}
\theoremstyle{definition}

\newtheorem*{defn*}{Definition}

\usepackage[lite]{amsrefs}
\newcommand{\Zbl}[1]{Zbl\,{#1}}

\begin{document}

\title[Periods of Eisenstein series of orthogonal groups (even primes)]
  {Compact periods of Eisenstein series of orthogonal groups of rank one at
   even primes}
\author{Jo\~ao Pedro Boavida}
\address{Departamento de Matem\'atica\\
         Instituto Superior T\'ecnico\\Universidade de Lisboa\\
         Av.\ Prof.\ Dr.\ Cavaco Silva\\2744--016 Porto Salvo, Portugal}
\thanks{Published in New York J. Math.\ \textbf{20} (2014), 153\ndash 181,
  available at \url{http://nyjm.albany.edu/j/2014/20-9.html}.
  Final submitted version, other than some citation updates.}
\email{joao.boavida@tecnico.ulisboa.pt}

\begin{abstract}
Fix a number field $k$ with its adele ring $\ade$.
Let $G=\gO(n+3)$ be an orthogonal group of $k$-rank $1$ and $H=\gO(n+2)$ a
$k$-anisotropic subgroup.  We have previously described how to factor the
global period
\begin{equation*}
  (E_\ph,F)_H = \int_{\Hq}E_\ph\cdot\cj F
\end{equation*}
of a spherical Eisenstein series $E_\ph$ of $G$ against a cuspform $F$
of $H$ into an Euler product.  Here, we describe how to evaluate the factors 
at even primes.  When the local field is unramified, we carry out the
computation in all cases.  We show also concrete examples of the complete
period when $k = \QQ$.  The results are consistent with the Gross--Prasad 
conjecture.
\end{abstract}

\maketitle

\section*{Introduction}

Fix a number field $k$; in some examples to follow, we take $k=\QQ$. 
Equip $k^{n+3}$ with a quadratic form $\langle\;,\;\rangle$ with matrix
\begin{equation*}
  \begin{pmatrix}1&&\\&*&\\&&-1\end{pmatrix}
\end{equation*}
with respect to the orthogonal decomposition $k^{n+3}=(k\cdot
e_+)\oplus k^{n+1}\oplus(k\cdot e_-)$.  Let $G=\gO(n+3)$ and its
subgroups act always on the right.  Let $H=\gO(n+2)$ be the fixer of $e_-$ 
and $\Th=\gO(n+1)$ be the fixer of both $e_+$ and $e_-$.  We consider only 
the case when $k^{n+2}$ is anisotropic; in particular, $G$ has $k$-rank $1$ 
and $H$ is $k$-anisotropic.

In this paper, we compute some automorphic periods associated to $G$ and $H$.
Such periods contain information about representations of those groups, as 
well as information of interest about certain $L$--functions.  The 
Gross--Prasad conjecture \cites{GrPr1, GrPr2, GrPr3} predicts that a 
representation of $\gO(n)$ occurs in a representation of $\gO(n+1)$ if and 
only if the corresponding tensor product $L$--function is nonzero on 
$\Re s = \frac12$.  The results we obtain are consistent with the
prediction.

Because $G$ is a reductive group, we can use its parabolics to organize the 
spectral decomposition of functions in $L^2(\Gq)$.  In our case, there is only
one parabolic up to conjugacy; in section \ref{s:setup}, we choose a 
representative $P$, with Levi component of the form $M \cong \Th \times \GL(1)$ 
and unipotent radical $N$.  Let also $K^G$ be some maximal compact; we will 
only consider right $K^G$--invariant functions, so-called \emph{spherical} 
functions.  We have two main families of spectral components (with more flags 
of parabolics we would have more).

The \emph{cuspidal components} decompose discretely, and are in some sense the 
analogue of the compact group components we have in $H$; we will not be much
concerned with them in this paper.

The \emph{Eisenstein series} are indexed by cuspidal components $\eta$ along 
$\Th$ and by characters $\om$ along $\GL(1)$.  These so-called \emph{Hecke
characters} are the analogue of Dirichlet characters in number fields other 
than $\QQ$.

We use $\theta$ to indicate an element of $\Th \subset G$ and
$m_\la$ for the element in $G$ corresponding to $\la \in \GL(1)$.  Given such 
$\eta$ and $\om$, we can extend
\begin{equation*}
  \ph(m_\la \theta)
  = \om(m_\la) \cdot \eta(\theta)
\end{equation*}
by left $\Na$--invariance and right $K^G$--invariance.  We will sometimes
use $\ph_{\om,\eta}$; other times $\ph_\om$ (when $\eta=1$); other times 
simply $\ph$.  We define the 
Eisenstein series $E_\ph = E_{\om, \eta}$ as the meromorphic continuation of
\begin{equation*}
  \sum_{\ga \in \Pk\lqu\Gk} \ph(\ga g).
\end{equation*}

The characters $\om$ are indexed in particular (but in number fields other than
$k=\QQ$, not only) by a continuous parameter $s$ chosen along 
$\Re s = \frac 12$ and appearing in the form $\de_P^s$, where $\de_P$ is 
the modular function of the parabolic $P$ and $\frac12$ has to do with 
normalizations of Haar measures.

Finally, the sum defining an Eisenstein 
series converges only for sufficiently large $\Re s$.  Therefore, we must 
also include any residues to the right of $\Re s = \frac12$ in the spectral
decomposition.

Let $E_\ph=E_{\om,\eta}$ be the spherical Eisenstein series associated with 
the Hecke character $\om:\GL(1)\to\Cx$ and the cuspidal component $\eta$ on 
$\Th$, and let $F$ be cuspidal on $H$.  We are interested in the period
\begin{equation*}
  (E_\ph,F)_H
  = \int_{\Hq} E_\ph\cdot \cj F.
\end{equation*}
(Maybe some aspects of this computation may guide the corresponding computation 
for periods along $H$ of cuspidal components of $G$, but the attempt must be 
left for another occasion.)

These same periods (called there global Shintani functions) were used by Katu,
Murase, and Sugano \cite{KaMuSu, MuSu} to obtain and study integral expressions 
for standard $L$--functions of the orthogonal group.  We already mentioned the 
Gross--Prasad conjecture.  Ichino and Ikeda \cite{IcIk} discuss further details 
and broader context is provided in papers by Gross, Reeder \cite{GrRe}, Jacquet, 
Lapid, Offen, and/or Rogawski \cites{JaLaRo,LaRo,LaOf}, Jiang \cite{Ji},
and Sakellaridis and Venkatesh \cites{Sak,SaVe}.

Both the uncorrected global period and all correction factors computed so
far are indeed nonzero.

Using the Phragm\'en--Lindel\"of principle, it is often possible to obtain 
\emph{convex} bounds for asymptotics of moments of automorphic $L$--functions.
The Lindel\"of hypothesis (a consequence of the Riemann hypothesis) yields
significantly better bounds, but any \emph{subconvex} bounds are of interest.
Iwaniec and Sarnak \cite{IwSa} survey important ideas about $L$--functions,
including subconvexity problems.

Diaconu and Garrett \cite{DiGa1, DiGa2} used a specific spectral identity to 
obtain subconvex bounds for second moments of automorphic forms in $\GL(2)$ 
over \emph{any} number field $k$.  That strategy has been explored in other 
other papers by them and/or Goldfeld \cite{DiGa2, DiGa3, DiGaGo} and used by 
Letang \cite{Le}.  In another paper \cite{Bo2} (from which, incidentally, this
paragraph and the third before it are taken almost verbatim), this author has 
applied that strategy to the periods discussed here to obtain a spectral 
identity for second moments of Eisenstein series of $G$.

Elsewhere~\cite{Bo}, the author has discussed how to factor the period
\begin{equation*}
  (E_\ph,F)_H
  = \int_{\Hq} E_\ph\cdot \cj F
\end{equation*}
into an Euler product and how to compute its local factors at odd primes.  
For the reader's convenience, we recapitulate those results as briefly 
as possible.

Because $H\supset\Th$ is a Gelfand pair \cite{AiGoSa, KaMuSu} and $\eta$ is 
spherical, we have
\begin{equation*}
  \int_{\Sq}\eta(\theta)\cdot\cj F(\theta h)\dd\theta
  = (\eta,F)_\Th \cdot f(h),
\end{equation*}
where $f$ is a spherical vector of $\Ind_\Th^H1$ normalized by $f(1)=1$.
Let $\pi=\otimes_v\pi_v$ be the irreducible representation generated by $f$
and $\om = \otimes_v \om_v$.
Letting $f_v$ be a generator of $\pi_v$ normalized by $f_v(1)=1$, the global
period is
\begin{equation}\label{e:global}
  (E_\ph,F)_H
  = (\eta,F)_\Th \cdot \int_{\Sa\lqu\Ha}\ph_\om\cdot f
  = (\eta,F)_\Th \cdot C_f \cdot \prod_v \int_{\Sv\lqu\Hv}\ph_{\om,v}\cdot f_v,
\end{equation}
for some constant $C_f$ (which is $1$ when $F=1$).

Because $\ph_{\om,v}$ and $f_v$ are spherical, and $\ph_v(1)=f_v(1)=1$, the
local integral is simply $\vol(\Sv\lqu\Hv)$ at anisotropic places $v$.

At isotropic places, we consider some parabolic $Q_v \subset H_v$.
If the period is nonzero, then $\pi_v$ is a quotient of a degenerate 
unramified principal series with respect to the Levi component of $Q_v$ and 
with parameter $\be_v$.

Let $\De$ be the discriminant of the restriction of $\langle\;,\;\rangle$ to 
$k^{n+2}$.  In the preceding paper \cite{Bo}, we determined the local factors 
at odd primes.  In this paper, we discuss what happens at even primes.

In section~\ref{s:setup}, we describe in more detail the conventions that
we used at odd primes and will adapt to the even primes; in particular, we
introduce the function $X$ and show how the local period may be readily 
obtained from it.

In section~\ref{s:even}, we explain what the required adaptations are and 
articulate a general method to determine the function $X$ at an even prime,
based on the number of solutions of an equation on finitely many finite rings
of characteristic $2$.

In section~\ref{s:dyadic}, we discuss two methods to count such solutions.
The more flexible of them, however, is only applicable when the local fields is
unramified.

In the subsequent sections, we apply either of those methods to each of the 
anisotropic forms, thus obtaining the function $X$ associated to
that form.  (When we apply the second method, the computation is 
limited to the unramified case.)

In all computations, it will transpire that only the dimension 
of the anisotropic component, the Hasse--Minkowski invariant, whether the 
discriminant is a unit, and (when it is a unit) its quadratic defect, play a 
role in the outcome.

To obtain the complete periods, we would need to know the factors at all
places, including the ramified cases not established here and the archimedean
places where the form is isotropic.  But if we restrict ourselves to $k = \QQ$, 
we have no ramification at the even prime and we can choose an anisotropic 
form at the archimedean place.  In section \ref{s:examples}, we combine all 
results established so far to compute the global period of the standard form 
with signature $(n+2, 1)$ in $k = \QQ$.

\subsection*{Acknowledgements}

This paper is a followup on the author's doctoral dissertation, done 
under the supervision of Paul Garrett.  As was the case with the previous 
paper, it is influenced by discussions with and talks by him.

The author wishes also to acknowledge a referee's advice on making the paper
more interesting to a wider audience.

\subsection*{Dedication.}

To the founder of IEEE--IST Academic (who went on to create IEEE
Academic) and his accomplice in the darkest hour.

\tableofcontents

\section{Setup (isotropic places)}\label{s:setup}

Let us recapitulate, from the previous paper~\cite{Bo}, 
what happens at isotropic places.

Recall we fixed a number field $k$ with adele ring $\ade$, 
and a quadratic form $\langle\;,\;\rangle$ with matrix
\begin{equation*}
  \begin{pmatrix}1&&\\&*&\\&&-1\end{pmatrix}
\end{equation*}
with respect to the decomposition $k^{n+3}=(k\cdot e_+)\oplus
k^{n+1}\oplus(k\cdot e_-)$.  We set $e=e_++e_-$ and named the
following groups of isometries:
\begin{align*}
  G&=\gO(n+3), &&\text{the isometry group of }
    \begin{pmatrix}*&*&*\end{pmatrix};\\
  H&=\gO(n+2), &&\text{the isometry group of }
    \begin{pmatrix}*&*&\phantom{*}\end{pmatrix};\\
  \Th&=\gO(n+1), &&\text{the isometry group of }
    \begin{pmatrix}\phantom{*}&*&\phantom{*}\end{pmatrix}.
\end{align*}
Let $P \subset G$ be the $k$-parabolic stabilizing $k\cdot e$.
The modular function of $P$ is given by $\de_P(p)=\abs t^{n+1}$ when 
$e\cdot p=e/t$.  In particular, $\de_P^s(p)=\abs t^\al$, with 
$\al=(n+1)s$.

We now choose an isotropic place $v$ which, from this point onward, we will
omit whenever possible.  Therefore, $k$ is the local field, $\oo$ is its
ring of integers, $\wg$ is a local uniformizer, and $\wg = q^{-1}$ ($q$ is 
the cardinality of the residue field).

\subsection*{Measure on $\Th\lqu H$}

Choose a hyperbolic pair $x$, $x'$ in $k^{n+2}$ so that 
$e_+\in (k\cdot x) \oplus (k\cdot x')$ and change coordinates so that
the restricted quadratic form has matrix 
\begin{equation*}
  \begin{pmatrix}&&1\\&B&\\1&&\end{pmatrix}
\end{equation*}
with respect to the orthogonal decomposition 
$k^{n+2}=(k\cdot x')\oplus k^n\oplus(k\cdot x)$.

Let $Q\subset H$ be the parabolic stabilizing the line $k\cdot x$; we have
\begin{equation*}
  \int_{\Th\lqu H} \text{function}(h)\dd h
  = \int_{\Th\lqu\Th Q} \text{function}(q)\dd q
  = \int_{(\Th\cap Q)\lqu Q} \text{function}(q)\dd q.
\end{equation*}
Here, $\Th\cap Q=\gO(n)$ is the fixer of $x$ and $x'$.  Set
\begin{equation*}
  m_\la = \begin{pmatrix}\la&&\\&\id&\\&&1/\la\end{pmatrix}
\qquad\text{and}\qquad
  n_a = \begin{pmatrix}1&a&-\frac12B(a)\\&\id&*\\
     \phantom{\frac12B(a)}&&1\end{pmatrix}.
\end{equation*}
With $M_* = \set{m_\la}$, we have
$\Th\cap Q=\set{\psm[.]{1&&\\&*&\\&&1}}$, $N^Q=\set{n_a}$, 
$M^Q=(\Th\cap Q)\cdot M_*$, and $Q=M^Q\cdot N^Q$.
The elements of $(\Th\cap Q)\lqu Q$ can be expressed as $n_a\cdot m_\la$
and $\de_Q(m_\la)=\abs\la^n$.  Moreover,
\begin{equation*}
  \dd{(n_a\cdot m_\la)}
  = \dd a\dd\la
\end{equation*}
(with $\dd\la$ multiplicative and $\dd a$ additive) is a 
\emph{right}-invariant measure.  Therefore, up to a multiplicative constant
independent of the integrand,
\begin{equation}\label{e:quo}
  \int_{\Th\lqu H} \text{function}(h)\dd h
  = \int_{\kx}\int_{k^n} \text{function}(n_a\cdot m_\la)\dd a\dd\la.
\end{equation}

\subsection*{Construction of $\ph_\om$}

We saw in \eqref{e:global} that the local factor is
\begin{equation*}
  \int_{\Th\lqu H} \ph_\om\cdot f,
\end{equation*}
where $f$ generates an unramified principal series; in fact, 
$f(m_\la)=\abs\la^\be$, for some $\be$ (again, we are omitting the place $v$).

We restrict ourselves to the non-archimedean places.

We choose coordinates preserving the decomposition 
$k^{n+2}=(k\cdot x')\oplus k^n\oplus(k\cdot x)$ from above, and let
$K^H \subset H$ be the compact open subgroup stabilizing integral (with 
respect to those coordinates) vectors.  The details of what coordinates
are actually chosen will transpire along the computation.

We want $\ph_\om$ to be associated with $\de_P^s$; that is, if $e\cdot p=e/t$
and $\al=(n+1)s$, then $\ph_\om(p)=\om(t)=\abs t^\al$.  Therefore, with $\Phi$ 
being the characteristic function of $\oo^{n+3}$, we define
\begin{equation*}
  \psi(g)
  = \int_{\kx}\om(t)\cdot\Phi(te\cdot g)\dd t
  \qquad\text{and}\qquad
  \ph_\om
  = \frac{\psi}{\psi(1)}.
\end{equation*}
The measure in $\kx$ is invariant with respect to multiplication normalized
so that $\ox$ has volume $1$.

\subsection*{Some notation}

Let $\abs\wg = q^{-1}$ and $\abs t=q^{-T}$.  We will also use $a=q^{-\al}$, 
$z=q^{-\be}$, and $w=zq^{-1}=q^{-\be-1}$ and write (with the same measure as
just above)
\begin{equation*}
  Z(\al) = \int_{\kx\cap\oo}\abs t^\al\dd t
  = \frac1{1-q^{-\al}}
  = \frac1{1-a}.
\end{equation*}
The integral converges only when $\Re\al$ is sufficiently large, but we will
use $Z(\al)$ to denote the meromorphic continuation.

With $z=q^{-\be}$, we define
\begin{align*}
  X_\ell^B(\rho)
  &=\meas\set{a\in\oo^n:\frac{B(a)-\rho}2=0 \mod\wg^\ell};\\[2\jot]
  X^B(\beta;\rho)
  &=\sum_{\ell\ge0}z^\ell X_\ell^B(\rho);\qquad\text{and}\\
  \Pi^B(\al,\be)
  &= \int_{\kx\cap\oo}\abs t^\al \; X^B(\be;t^2)\dd t.
\end{align*}
(When there is no risk of ambiguity, we suppress $B$ or $\rho$, 
or use $n$ instead of $B$.)

In order to make clear what adaptations are needed at even primes, we must
repeat the following two proofs, with minor adjustments.

\begin{prop*}
Up to a multiplicative constant independent of the integrand,
\begin{equation*}
  \int_{\Th\lqu H} \psi \cdot f
  = \int_{\Th\lqu H}\int_{\kx}
    \abs t^\al \cdot\Phi(te\cdot h)\cdot f(h)\dd t\dd h
  = \Pi^n(\al-\be-n,\be).
\end{equation*}
(This is valid at all non-archimedean primes.)
\end{prop*}

\begin{proof}
According to \eqref{e:quo}, we have
\begin{equation*}
  \int_{\Th\lqu H}\int_{\kx}
    \abs t^\al \; \Phi(te\cdot h) \; f(h)\dd t\dd h
  = \int_{\kx}\int_{\kx}\int_{k^n}
    \abs t^\al \; \abs\la^\be \; \Phi(te\cdot n_a\cdot m_\la)\dd a\dd\la \dd t.
\end{equation*}

At this point, we specify $e_+=x'+\frac12x$.  Noting that
\begin{equation*}
  e\cdot n_a\cdot m_\la
  = (e_+ + e_-)\cdot n_a\cdot m_\la
  = e_+\cdot n_a\cdot m_\la + e_-,
\end{equation*}
we have (in $k^{n+2} = (k\cdot x') \oplus k^n \oplus (k\cdot x)$)
\begin{equation*}
  e_+\cdot n_a\cdot m_\la
  = \begin{pmatrix}1&0&\tfrac12\end{pmatrix}\cdot n_a\cdot m_\la
  = \begin{pmatrix}\la&a&\tfrac1{2\la}\bigl(1-B(a)\bigr)\end{pmatrix}
\end{equation*}
and (in $k^{n+3} = k^{n+2} \oplus (k\cdot e_-)$)
\begin{equation*}
  te\cdot n_a\cdot m_\la
  = \Bigl(\la t,at,\tfrac1{2\la t}\bigl(t^2-B(at)\bigr),t\Bigr).
\end{equation*}

Therefore, after a change of variables,
\begin{equation*}\begin{split}
  \int_{\Th\lqu H}\psi \cdot f
  &=\int_{\kx}\int_{\kx}\int_{k^n}
      \abs t^{\al-\be-n} \; \abs\la^\be \; \Phi
      \Bigl(\la,a,\tfrac1{2\la}\bigl(t^2-B(a)\bigr),t\Bigr)
      \dd a\dd\la\dd t\\
  &=\int_{\kx\cap\oo}\abs t^{\al-\be-n}
      \int_{\kx\cap\oo}\abs\la^\be
      \int_{\oo^n}\ch_\oo \Bigl(\frac{t^2-B(a)}{2\la}\Bigr)
      \dd a\dd\la\dd t\\
  &=\int_{\kx\cap\oo}\abs t^{\al-\be-n}
      \sum_{\ell\ge0} z^\ell X^n_\ell(t^2)
  = \int_{\kx\cap\oo}\abs t^{\al-\be-n} X^n(\be; t^2).
  \qedhere
\end{split}\end{equation*}
\end{proof}

\begin{prop}\label{p:odd:lpi}
Fix $\ph_\om(p) = \de_P(p)^s = \om(t) = \abs t^\al$ when $e\cdot p = e/t$ and
$\al = (n+1) s$.  Fix also a cuspidal $F$ generating an irreducible 
$\pi = \otimes_v \pi_v$.  Let $f_v$ be a generator of $\pi_v$ normalized vy
$f_v(1) = 1$, and let $\be_v$ be the local parameter of the unramified 
principal series representation $\pi_v$.

The local factor at the non-archimedean place $v$ (ommited in the remainder
of the statement) in the global period $(E_\ph, F)_H$ is
\begin{equation*}
  \int_{\Th\lqu H}\ph_\om\cdot f
  = \frac1{\psi(1)}\int_{\Th\lqu H}\psi\cdot f
  = \frac{\Pi^n(\al-\be-n,\be)}{\abs2^\al \; Z(\al)}
\end{equation*}
up to a multiplicative constant.
\end{prop}

For the odd prime case, the multiplicative constant was determined in the 
previous paper, but the method does not seem applicable to even primes.  In the 
cases for which we have computed $\Pi$, it depends only on the dimension of
the anisotropic component and Witt index, the discriminant (whether it is a 
unit or is quadratic defect), and (for even primes) the Hasse--Minkowski 
invariant.

\begin{proof}
The multiplicative constant accounts for the normalization implied in the 
integral \eqref{e:quo}.  Additionally,
\begin{equation*}
  \psi(1)
  = \int_{\kx} \abs t^\al \; \Phi(te) \dd t
  = \int_{\kx} \abs t^\al \; \Phi\bigl(t,0,\tfrac t2,t\bigr) \dd t
  = \int_{\kx\cap2\oo} \abs t^\al
  = \abs2^\al \; Z(\al).
  \qedhere
\end{equation*}
\end{proof}

\subsection*{Dimension reduction}

By taking hyperbolic planes away, we can simplify the evaluation of
\thref{p:odd:lpi} significantly.  In fact, if there is a hyperbolic 
subspace with dimension $2k$ and $n=m+2k$, then
\begin{align*}
  X^{m+2k}(\be;\rho)
  &=\frac{Z(\be+1)}{Z(\be+k+1)}\cdot X^m(\be+k;\rho),\\
  \Pi^{m+2k}(\al,\be)
  &=\frac{Z(\be+1)}{Z(\be+k+1)}\cdot\Pi^m(\al,\be+k).
\end{align*}
This is valid at all non-archimedean places.

All that is left to do, is to find the functions $X$ and $\Pi$ for anisotropic
forms.  The odd prime case was discussed in the previous paper.  For even 
primes, we have anisotropic forms in $k^m$ with $m \le 4$.

\section{Even primes}\label{s:even}

The actual computation of $X$ and $\Pi$ (for anisotropic forms) at odd primes
relied substantially on an anisotropy lemma, which guaranteed that certain
equations had no solution modulo $\wg^\ell$.  For even primes, we rely
on a similar lemma.

In all that follows, $e = \ord 2$ is the ramification index and 
$B(x) = \sum_i a_i x_i^2$ is a form with $0 \le \ord a_i \le 1$.

\begin{lem}\label{l:even:aniso}
Let $B(x) = \sum_i a_i x_i^2$ be anisotropic.  Then, for each $x \ne 0$,
\begin{equation*}
  \max_i \abs{a_i x_i^2} \ge \abs{B(x)} \ge \max_i \abs{4 a_i x_i^2}.
\end{equation*}
\end{lem}

Were that not the case, then the following lemma, with $\rho = B(x)$, would
yield a nonzero solution of $B(x) = 0$.

\begin{lem}\label{l:even:sol}
Let $B(x) = \sum_i a_i x_i^2$ be a form satisfying $0 \le \ord a_i \le 1$.
If there is a nonzero $x \in \oo^n$ and $\ell > \ord(4a_ix_i^2)$ (for 
some $i$) such that
\begin{equation*}
  B(x) = \rho \mod \wg^\ell,
\end{equation*}
and if $\abs{a_j x_j} = \max_i \abs{a_i x_i}$, then
there is $y \in x + (2a_jx_j)^{-1}\wg^{\ell+1}\oo^n$ such that $B(y) = \rho$.  
In fact, $X^B_{\ell+1}(\rho) = \abs\wg \; X^B_\ell(\rho)$ if 
$\ell > \ord(2\rho)$.
\end{lem}

\begin{proof}
This follows from some versions of Hensel's lemma.  We prove only the exact
details we need in the continuation, as we will rely on specifics of the
dyadic case.

Choose the highest $H \ge 0$ such that $x \in \wg^H\oo^n$ and 
write $x = \wg^H x'$, $y = \wg^H y'$, $\rho = \wg^{2H} \rho'$, and 
$\ell = 2H + \ell'$.  Replacing $x$, $y$, $\rho$, and $\ell$ by $x'$, $y'$,
$\rho'$, and $\ell'$, both in the statement and in the conclusions, we see
that we need address only the case $H=0$, that is, the case $\abs{x_j} = 1$.

Write $r = (2a_j)^{-1}\wg^{\ell+1}$ and $y = x + ru$, leaving
$u \in \oo^n$ unspecified.  We have
\begin{equation*}
  \sum_i a_i y_i^2
  = \sum_i a_i (x_i + r u_i)^2
  = \sum_i a_i x_i^2 + \sum_i 2 r a_i x_i u_i + \sum_i r^2 a_i u_i^2.
\end{equation*}
Because $\abs{a_j} = \abs{a_j x_j^2} = \max_i \abs{a_i x_i^2}$, we have
$\abs{4a_j} = \abs{4a_jx_j^2} \ge \abs{4a_ix_i^2} > \abs{\wg^\ell}$, so
\begin{equation*}
 \abs{r^2 a_i u_i^2}
 = \abs{(2a_j)^{-2} \wg^{2\ell+2} a_i u_i^2}
 = \abs*{\frac{\wg^\ell}{4a_j} \frac{\wg a_i}{a_j} \wg^{\ell+1} u_i^2}
 < \abs{\wg^{\ell+1}},
\end{equation*}
Therefore, none of the $r^2 a_i y_i^2$ summands contributes 
modulo $\wg^{\ell+1}$.  On the other hand, from
$\abs{a_ix_i} \le \abs{a_j}$ we obtain
\begin{equation*}
  \abs{2r a_i x_i u_i}
  = \abs{a_j^{-1} \wg^{\ell+1} a_i x_i u_i}
  = \abs*{\frac{a_i x_i}{a_j} \wg^{\ell+1} u_i}
  \le \abs{\wg^{\ell+1}},
\end{equation*}
with equality (at least) when $i=j$ and $\abs{u_j} = 1$.

That means that, no matter the choice for the other $u_i$, we can use $u_j$ 
to control the value of $B(y)$ modulo $\wg^{\ell+1}$.  In other words, for 
exactly $\abs\wg$ (that is, one $q$th) of the choices of $u_j$ (corresponding 
to $\abs\wg$ the possible refinements of $x$), we obtain
\begin{equation*}
  B(y) = \rho \mod \wg^{\ell+1},
\end{equation*}
a refinement of our original equation.  Taking a limit, we 
obtain the desired solution.

If we know $\ell > \ord(2\rho)$, we need no specifics on the values of $x_i$,
so we can conclude $X^B_{\ell+1}(\rho) = \abs\wg \; X^B_\ell(\rho)$.  The
apparent mismatch between this statement and what was done above is due
to the definition of $X_\ell$ involving an equation modulo $2\wg^\ell$.
\end{proof}

With $T = \ord t \ge 0$, suppose now that $\rho = 4t^2 \ne 0 \bmod 2\wg^\ell$
(so, $\ell - 2T > e$), and that $x$ is a solution of 
$B(x) = \rho \bmod 2\wg^\ell$.  Then, according to lemma~\thref{l:even:aniso}, 
$\abs{4t^2} \ge \max_i \abs{4a_i x_i^2}$; hence, $\abs t \ge \abs{x_i}$ and 
$x \in t\oo^n$.  Therefore,
\begin{equation*}\begin{split}
  X_\ell(4t^2)
  &= \meas\set{tx \in t\oo^n : B(tx) = 4t^2 \mod 2\wg^\ell}\\
  &= \abs t^n \meas\set{x \in \oo^n : B(x) = 4 \mod 2\wg^{\ell-2T}}
  = \abs{\wg^n}^T X_{\ell-2T}(4).
\end{split}\end{equation*}
Still with $\ell > 2T + \ord2$, we also have, by the same reasoning,
\begin{equation*}
  X_\ell(0)
  = \meas\set{tx \in t\oo^n : B(tx) = 4t^2 \mod 2\wg^\ell}
  = \abs{\wg^n}^T X_{\ell-2T}(0).
\end{equation*}
Finally, observing that $4t^2 = 0 \bmod 2\wg^\ell$ for $\ell \le 2T + e$
and using $z = q^{-\be}$ and $u = z^2q^{-n}$, we obtain
\begin{equation*}\begin{split}
  X(\be; 4t^2) - X(\be; 0)
  &= \sum_{\ell > 2T + e} z^\ell X_\ell(4t^2)
     - \sum_{\ell > 2T + e} z^\ell X_\ell(0)\\
  &= z^{2T} \abs{\wg^n}^T \sum_{k > e} z^k\bigl(X_k(4) - X_k(0)\bigr)\\
  &= u^T \sum_{k \ge 0} z^k\bigl(X_k(4) - X_k(0)\bigr)
  = u^T \bigl(X(\be; 4) - X(\be; 0)\bigr).
\end{split}\end{equation*}
We have thus proven
\begin{equation*}
  X(\be; 4\wg^{2T}) = X(\be; 0) + u^T X(\be; 4) - u^T X(\be; 0),
\end{equation*}
which leads to this conclusion:

\begin{prop}\label{p:even:Pi}
If $u = z^{-2\be - n}$ and $a = q^{-\al}$, then
\begin{equation*}
  \Pi(\al, \be)
  = \sum_{0 \le T < e} a^T X(\be; \wg^{2T})
    + \frac{\abs2^\al}{1-au} X(\be; 4)
    + \frac{\abs2^\al (a-au)}{(1 - a) (1 - au)} X(\be; 0).
\end{equation*}
\end{prop}

\begin{proof}
We have
\begin{equation*}
  \Pi(\al, \be)
  = \sum_{0 \le T < e} a^T X(\be; \wg^{2T})
    + \sum_{T\ge0} a^{T+e} X(\be; 4\wg^{2T}).
\end{equation*}
But
\begin{equation*}\begin{split}
  \sum_{T\ge0} a^T X(\be; 4\wg^{2T})
  &= \sum_{T\ge0} a^T X(\be; 0)
    + \sum_{T\ge0} (au)^T X(\be; 4)
    - \sum_{T\ge0} (au)^T X(\be; 0)\\
  &= \frac{1}{1 - au} X(\be; 4) + \frac{(a - au)}{(1 - a)(1 - au)} X(\be; 0).
  \qedhere
\end{split}\end{equation*}
\end{proof}

Combining this proposition with lemma~\thref{l:even:sol}, we see that only
finitely many values $X_\ell(t^2)$ need be computed.

Indeed, choose $t$ with $1 \ge \abs t = q^{-T} \ge \abs2$
(this argument works for any $T \ge 0$).  If $\ell > \ord(2t^2) = 2T + e$ 
(so, at least for $\ell > \ord8$) and $k \ge 0$, the lemma tells us that 
$X_{\ell+k}(t^2) = \abs\wg^k X_\ell(t^2)$.  Therefore,
\begin{equation*}
  X(\be; t^2)
  = \sum_{0 \le \ell \le 2T + e} z^\ell X_\ell(t^2)
    + \sum_{k \ge 0} z^{2T+e+k+1} \abs\wg^k X_{2T+e+1}(t^2),
\end{equation*}
which, with $w=zq^{-1}$, simplifies to
\begin{equation*}
  X(\be; t^2)
  = \sum_{0 \le \ell < 2T + e + 1} z^\ell X_\ell(t^2)
    + \frac{z^{2T+e+1}}{1 - w} X_{2T+e+1}(t^2).
\end{equation*}

The anisotropy lemma~\thref{l:even:aniso} yields a similar reduction for
$X_\ell(0)$: if $x \in \oo^n$ and $B(x) = 0 \bmod 2\wg^\ell$, then it must be
that $\abs{2\wg^\ell} \ge \max_i\abs{4a_ix_i^2}$, or 
$\abs{2a_ix_i^2} \le \abs{\wg^\ell}$.  If 
$\abs{\wg^\ell} = \abs{2\wg^{2k+1}}$ or $\abs{\wg^\ell} = \abs{2\wg^{2k}}$,
we may rely on $\abs{x_i}\le\abs{\wg^k}$, and in either case
\begin{equation*}
  X_\ell(0)
  = \meas\set{\wg^k x \in \wg^k\oo^n : B(\wg^k x) = 0 \mod 2\wg^\ell}
  = \abs{\wg^n}^k X_{\ell-2k}(0),
\end{equation*}
leading us to
\begin{equation*}\begin{split}
  \sum_{\ell \ge e} z^\ell X_\ell(0)
  &= \sum_{k \ge 0} z^{e+2k} X_{e+2k}(0)
    + \sum_{k \ge 0} z^{e+2k+1} X_{e+2k+1}(0)\\
  &= z^e \sum_{k \ge 0} z^{2k} \abs{\wg^n}^k X_e(0)
     + z^{e+1} \sum_{k \ge 0} z^{2k} \abs{\wg^n}^k X_{e+1}(0).
\end{split}\end{equation*}
Using again $u=z^2 q^{-n}=q^{-2\be-n}$, we obtain
\begin{equation*}
  X(\be; 0)
  = \sum_{0 \le \ell < e} z^\ell X_\ell(0)
    + \frac{z^e}{1 - u} X_e(0) + \frac{z^{e+1}}{1 - u} X_{e+1}(0).
\end{equation*}

In summary:

\begin{prop}\label{p:even:X}
If $z=q^{-\be}$, $w=zq^{-1}$, $u=z^2q^{-n}$, and $T \ge 0$, then
\begin{align*}
  X(\be; \wg^{2T})
  &= \sum_{0 \le \ell < 2T + e + 1} z^\ell X_\ell(\wg^{2T})
    + \frac{z^{2T+e+1}}{1 - w} X_{2T+e+1}(\wg^{2T});\\
  X(\be; 0)
  &= \sum_{0 \le \ell < e} z^\ell X_\ell(0)
    + \frac{z^e}{1 - u} X_e(0)
    + \frac{z^{e+1}}{1 - u} X_{e+1}(0).
\end{align*}
\end{prop}

We note that many of these values are repeated.  For example, if $\ell \le e$,
then $X_\ell(0) = X_\ell(4)$.  Additionally, if $\min_i \abs{a_i} = 1$
(that is, all coefficients of the diagonal quadratic form are units), then 
the anisotropy lemma~\thref{l:even:aniso} implies 
$X_{e+1}(0) = \abs{\wg^n} X_{e-1}(0)$.

In practice, what we shall do is determine $X(\be; \wg^{2T})$ for all $T$
when it takes no more effort than to do so only for $T \le e$, or resort
to proposition~\thref{p:even:X} otherwise.

\section{Conics in dyadic fields}\label{s:dyadic}

The computation of each $X_\ell(t^2)$ amounts to counting points modulo
$2\wg^\ell$ in conics.  We discuss some preliminaries first.

We rely substantially on O'Meara's~\cite[\S63]{Om} discussion of the quadratic 
defect in dyadic fields.  We recall some of the relevant facts.
The quadratic defect of $\rho$ is the intersection of all ideals $b\oo$ for 
$b$ such that $\rho - b$ is a square.  If $b\oo$ is the quadratic defect of
$\rho$, then $\eta^2 b\oo$ is the quadratic defect of $\eta^2\rho$.
If $\ord\rho$ is odd, then the quadratic defect of $\rho$ is $\rho\oo$.  If 
$\ord\rho$ is even, then the quadratic defect of $\rho$ is $0$ (if $\rho$ is 
a square) or $4\rho\oo$, or $\rho\wg^{2k+1}\oo$ with $0 \le k < e$.
If $\rho = \eta^2 + b$ is a unit and $0 < \ord b = 2k + 1 < 2e$ or 
$0 < \ord b = 2k < 2e$, then the quadratic defect of $\rho$ is $\wg^{2k+1}\oo$.
The quotient of two units with quadratic defect $4\oo$ is a square.  (Hence,
half the units of the form $\eta^2 + b$ with $\ord b = 2e$ are squares, and
the other half have quadratic defect $4\oo$.)  If $\rho = \eta^2 + b$ is a
unit and $\ord b > 2e$, then $\rho$ is a square.

We recall that, for fixed dimension $m$, a form is classified by its 
discriminant~$\De$ (we include the sign $(-1)^{\lfloor m/2\rfloor}$ in
its definition) and its Hasse--Minkowski invariant, built from the Hilbert
symbol $(\;, \;)$.

\begin{lem}\label{l:dyadic:hmi}
Fix a non-square unit $\De$.

If the quadratic defect of $\De$ is $4\oo$, then $(a, \De) = (-1)^{\ord a}$.

Otherwise, there is a unit $a$ with quadratic defect $\wg\oo$ such that
$(a, \De) = -1$.
\end{lem}

\begin{proof}
The first claim is proved by O'Meara~\cite[\S63]{Om}.

For the second claim, let $\wg^d\oo$ be the quadratic defect of 
$\De = 1+\wg^d v$, as any other unit $\De$ with the same quadratic defect may 
be obtained with a change of variable in $y$.  Write $a = 1 + \wg u$, 
with $u$ a unit.  We want to show
\begin{equation*}
  ax^2 + \De y^2
  = x^2 + \wg ux^2 + y^2 + \wg^d vy^2
  = (x + y)^2 - 2xy + \wg ux^2 + \wg^d vy^2
\end{equation*}
is never a square, unless $x = y = 0$.  If $\abs x \ne \abs y$, then the 
quadratic defect of the sum is the largest of $\wg^d y^2\oo$ and $\wg x^2\oo$.  
Therefore, we need only check the case $\abs x = \abs y$.

In the unramified case, use $\wg = 2$ and $d = 1$.  Let also $y = xt$.
Then we want to choose $u$ so that
\begin{equation*}
  ax^2 + \De y^2
  = (x + xt)^2 + 2x^2(u - t + vt^2)
\end{equation*}
is never a square, or, which is the same, so that $u - t + vt^2 \ne 0 \bmod2$
(were there any solutions of the latter equation, then we could refine at least
one of the two so as to obtain a solution of the former).

But $t - vt^2$ is a separable quadratic polynomial, so in a finite field
it takes only half the possible values, and we may choose for $u$ any of the
values it does not take.  The same reasoning shows that 
$ux^2 - xy + vy^2 = 0 \bmod 2$ has only one solution ($x = y = 0 \bmod2$) if
and only if $(1+2u, 1+2v) = -1$.

In the ramified case with $d > 1$, we see 
$\ord(- 2xy + \wg ux^2 + \wg^d vy^2) = \ord(\wg ux^2)$ is odd,
so $ax^2 + \De y^2$ is, indeed, never a square.

In the ramified case with $d = 1$, we may choose $a = \De$.  Indeed, the 
reasoning above shows that $(\De, -1) = -1$ whenever $d < e$, so 
$(\De, \De) = (\De, -1) = -1$.
\end{proof}

Here, we point out that if $u$ is a unit, to say $1 + 4u$ is not a square 
is to say there is no unit $v$ such that $(1 + 2v)^2 = 1 + 4v + 4v^2 = 1 + 4u$,
which is to say $v + v^2 - u = 0$ is impossible, or $(1+2u, -1) = -1$ in
the unramified case.

\subsection*{The first method}

In its crudest form, the question we wish to answer is how many solutions 
there are to
$x^2 = \rho \bmod \wg^\ell$.  Clearly, there are any if and only if $\rho$
is a square modulo $\wg^\ell$.  Most often, the number of solutions does not
depend further on $\rho$; in fact, if $\rho$ is not a square, then the second 
case listed below does not occur.

\begin{lem}\label{l:dyadic:main}
Let $X = \meas\{x \in \oo : x^2 = \rho \bmod \wg^\ell\}$, where $\rho \in \oo$ 
and $\ell \ge 0$.  Write $\rho = \eta^2 + b$, where $b\oo$ is the quadratic 
defect of $\rho$.

If $b \ne 0 \bmod \wg^\ell$, then $X = 0$.

If $b = 0 \bmod \wg^\ell$ and $\abs{\wg^\ell} < \abs{4\eta^2}$, then
$X = 2 \, \abs{\wg^\ell/2\eta}$.

Otherwise, $X = \abs\wg^{\lceil\ell/2\rceil}$.
\end{lem}

\begin{proof}
The first claim is a consequence of the definition of quadratic defect.

The case $b = 0 \bmod \wg^\ell$ remains.
We want to find solutions of $x^2 = \eta^2 \bmod \wg^\ell$, which we rewrite 
as $(x - \eta)(x + \eta) = 0 \bmod \wg^\ell$.

For the second claim, if $\abs{\wg^\ell} < \abs{4\eta^2}$, then $b = 0$ and 
$\rho$ is a square.  The options $\abs{x-\eta} = \abs{x+\eta} = \abs{2\eta}$ 
and $\abs{x-\eta} > \abs{2\eta}$ would lead to 
$\abs{(x\eta)(x+\eta)} \ge \abs{4\eta^2} > \abs{\wg^\ell}$.  Hence, in order
for $x$ to be a solution, we require $\abs{x\pm\eta} < \abs{2\eta}$, in which
case $\abs{x\mp\eta} = \abs{2\eta}$.  That is, we are requiring
$\abs{x\pm\eta} \ge \abs{\wg^\ell/2\eta} < \abs{2\eta}$.  Therefore,
$X = 2 \, \abs{\wg^\ell/2\eta}$.

For the third claim, we consider $\abs{\wg^\ell} \ge \abs{4\eta^2}$.  If
$\abs{x-\eta} \le \abs{2\eta}$, then $\abs{x+\eta} \le \abs{2\eta}$, so
$\abs{(x-\eta)(x+\eta)} \le \abs{4\eta^2} \le \abs{\wg^\ell}$, so $x$ is a 
solution.  If $\abs{x-\eta} > \abs{2\eta}$, then $\abs{x+\eta} = \abs{x-\eta}$,
so $\abs{(x-\eta)(x+\eta)} = \abs{x-\eta}^2$, so $x$ being a solution is 
equivalent to $\abs{x-\eta}^2 \le \abs{\wg^\ell}$.  Therefore, 
$X = \abs\wg^{\lceil\ell/2\rceil}$.
\end{proof}

We will compute several sums of the form
\begin{equation*}
  \sum_{0 \le \ell < L} z^\ell \abs\wg^{\lceil(\ell+o)/2\rceil}
  = \sum_{0 \le 2k < L} z^{2k} \abs\wg^{\lceil(2k+o)/2\rceil}
    + \sum_{0 \le 2k+1 < L} z^{2k+1} \abs\wg^{\lceil(2k+1+o)/2\rceil}.
\end{equation*}
Set $w = z \abs\wg$.  The first summand is
\begin{equation*}
  \sum_{0 \le k < \lceil L/2\rceil} (zw)^k \abs\wg^{\lceil o/2\rceil}
  = \abs\wg^{\lceil o/2\rceil} \frac{1 - (zw)^{\lceil L/2\rceil}}{1 - zw},
\end{equation*}
while the second is
\begin{equation*}
  z \abs\wg^{\lceil(o+1)/2\rceil} \frac{1 - (zw)^{\lfloor L/2\rfloor}}{1 - zw}
  = w \abs\wg^{\lfloor o/2\rfloor} \frac{1 - (zw)^{\lfloor L/2\rfloor}}{1 - zw}.
\end{equation*}
Therefore,
\begin{equation}\label{e:dyadic:sum}
  \sum_{0 \le \ell < L} z^\ell \abs\wg^{\lceil(\ell+o)/2\rceil}
  = \abs\wg^{\lceil o/2\rceil} \frac{1 - (zw)^{\lceil L/2\rceil}}{1 - zw}
    + w \abs\wg^{\lfloor o/2\rfloor}
      \frac{1 - (zw)^{\lfloor L/2\rfloor}}{1 - zw}.
\end{equation}

\subsection*{The second method}

Though it is somewhat more versatile, this method can be used only in the 
unramified case.  While discussing it, we always use $\wg = 2$.  We
fix two units $u$ and $v$ such that $(1+2u, 1+2v) = -1$; as discussed in the 
proof of lemma~\thref{l:dyadic:hmi}, this is equivalent to saying that all 
solutions of $ux^2 + xy + vy^2 = 0 \bmod2$ satisfy $x = y = 0 \bmod2$.

Fix a quadratic polynomial 
$P(x, y) = ux^2 + Cxy + vy^2 + Bx + Ay + D \in \oo[x, y]$ with $C = 1 \bmod2$.
Changing variables to $x = X + A$ and $y = Y + B$ and reducing modulo $2$, we 
obtain
\begin{equation*}
  P(x, y)
  = uX^2 + XY + vY^2 + P(A, B) \mod 2.
\end{equation*}

\begin{lem}\label{l:dyadic:proj}
If $P(A, B) = 0 \bmod 2$ and $\ell \ge 1$, then any solutions that may exist 
satisfy $x = A \bmod2$ and $y = B \bmod2$.

If $P(A, B) \ne 0 \bmod 2$ and $\ell \ge 1$, then
\begin{equation*}
  \meas\set{(x, y) \in \oo^2 : P(x, y) = 0 \mod2^\ell}
  = q^{-\ell} + q^{-\ell-1}.
\end{equation*}
\end{lem}

\begin{proof}
We replace $x = X + A$ and $y = Y + B$.  This has no effect on the measure.

When $\ell = 1$, the measure we want is
\begin{equation*}
  \meas\set{(X, Y) \in \oo^2 : uX^2 + XY + vY^2 = P(A, B) \mod 2}.
\end{equation*}

If $P(A, B) = 0 \bmod 2$, then $X = Y = 0 \bmod 2$.  Otherwise, 
$X = Y = 0 \bmod2$ is not a solution.  Given a representative $(X, Y)$ 
of a projective line (with respect to the residue field), there is exactly one 
(nonzero) value of $t \bmod2$ such that $(Xt, Yt)$ is a solution.  There are 
$q + 1$ projective lines, so the measure of the solution set is 
$(q + 1) / q^2 = q^{-1} + q^{-2}$.

For $\ell > 1$, suppose $(X, Y)$ is a solution modulo $2^\ell$ with $Y$ a 
unit.  Fix \emph{any} refinement of $Y$ modulo $2^{\ell+1}$.  The coefficient 
of degree $1$ in $P(X + A, Y + B) \in \oo[X]$ is a unit.  Therefore, exactly 
one $q$th of the refinements of $X$ modulo $2^{\ell+1}$ will yield a solution 
of the equation modulo $2^{\ell+1}$.  The corresponding argument may be made 
if $X$ is a unit.  The effect in either case is
\begin{equation*}
  \meas\set{\text{solutions modulo $2^{\ell+1}$}}
  = \abs\wg \cdot \meas\set{\text{solutions modulo $2^\ell$}}.
  \qedhere
\end{equation*}
\end{proof}

\section{Even primes---$m=0$}\label{s:even0}

If the original form is totally isotropic, we may reduce it to the case $m=0$.

\begin{prop}\label{p:ex:n:6}
Let $B = 0$ be the form in $0$ variables.  With $z = q^{-\be}$ and 
$\abs t = \abs\wg^T$, we have
\begin{equation*}
  X(\be; t^2)
  = \frac{1 - z^{2T+1-e}}{1 - z}
\end{equation*}
if $\abs{t^2} \le \abs2$, and $X(\be; t^2) = 0$ otherwise.
\end{prop}

\begin{proof}
We want to evaluate
\begin{equation*}
  X_\ell(t^2)
  = \meas \set{0 : 0 = t^2 \mod 2\wg^\ell}.
\end{equation*}
The equation holds exactly if $t^2 = 0 \bmod 2\wg^\ell$, that is, if
$2T \ge \ell+e$.
\end{proof}

\section{Even primes---$m=1$}\label{s:even1}

\begin{prop}\label{p:ex:n:5}
Let $B(x) = \De x^2$, where $\De$ is a nonsquare unit with quadratic defect 
$\wg^d\oo$.  With $z = q^{-\be}$, $w = zq^{-1}$, and $\abs t=q^{-T}$, we have
\begin{equation*}
  X(\be; t^2)
  = \abs\wg^{\lceil e/2\rceil} \,
    \frac{1 - (zw)^{T+\lceil(d+1-e)/2\rceil}}{1 - zw}
    + w \abs\wg^{\lfloor e/2\rfloor} \,
      \frac{1 - (zw)^{T+\lfloor(d+1-e)/2\rfloor}}{1 - zw}
\end{equation*}
if $\abs2 \ge \abs{\wg^d t^2}$, and $X(\be; t^2) = 0$ otherwise.
\end{prop}

\begin{proof}
According to lemma~\thref{l:dyadic:main},
\begin{equation*}
  X_\ell(t^2)
  = \meas\set{x \in \oo: \De x^2 = t^2 \mod 2\wg^\ell}
\end{equation*}
(we use $2\wg^\ell$ here, instead of $\wg^\ell$ there) fails to be $0$ only 
if $\wg^d t^2 = 0 \bmod 2\wg^\ell$ (a unit $\De$ and its inverse have the same
quadratic defect), that is, only 
when $0 \le \ell < 2T + d - e + 1$.  Using the final case of that lemma
and applying \eqref{e:dyadic:sum}, we obtain the answer.
\end{proof}

\begin{prop}\label{p:ex:n:7}
Let $B(x) = \De x^2$, where $\De$ is a unit square.  With $z = q^{-\be}$, 
$w = zq^{-1}$, and $\abs t = q^{-T}$, we have
\begin{equation*}
  X(\be; t^2)
  = \frac{\abs\wg^{\lceil e/2\rceil}
          + w \abs\wg^{\lfloor e/2\rfloor}}{1 - zw}
    - \frac{(1 + z) w^{e+1}}{1 - zw} (zw)^T
    + \frac{2 (zw)^T w^{e+1}}{1 - w}.
\end{equation*}
\end{prop}

\begin{proof}
According to lemma~\thref{l:dyadic:main}, $X_\ell(t^2)$ is different depending 
on whether $0 \le \ell < 2T+e+1$ or $\ell \ge 2T+e+1$.  In the first case,
we obtain exactly the same sum as in the previous proof, but with $d = 2e$.
Upon simplification, this yields the first two summands in the statement.  
For $\ell \ge 2T+e+1$, lemma~\thref{l:dyadic:main} tells us
\begin{equation*}
  \sum_{\ell \ge 2T + e + 1} z^\ell X_\ell(t^2)
  = \sum_{\ell \ge 2T + e + 1} z^\ell \, 2\abs{\wg^\ell/t}
  = \frac{2 (zw)^T w^{e+1}}{1 - w}.
  \qedhere
\end{equation*}
\end{proof}

\begin{prop}
Let $B(x) = \De x^2$, where $\abs\De = \abs\wg$.  With $z = q^{-\be}$, 
$w = zq^{-1}$, and $\abs t = q^{-T}$, we have
\begin{equation*}
  X(\be; t^2)
  = \abs\wg^{\lfloor e/2\rfloor} \,
    \frac{1 - (zw)^{T+1-\lceil e/2\rceil}}{1 - zw}
    + z \abs\wg^{\lceil e/2\rceil} \,
      \frac{1 - (zw)^{T-\lfloor e/2\rfloor}}{1 - zw}
\end{equation*}
if $\abs 2 \ge \abs{t^2}$, and $X(\be; t^2) = 0$ otherwise.
\end{prop}

\begin{proof}
$X_\ell(t^2)$ fails to be $0$ only if $t^2 = 0 \bmod2\wg^\ell$, that is,
only if $2T \ge \ell + e$, or $0 \le \ell < 2T - e + 1$.  In that case, 
lemma~\thref{l:dyadic:main} tells us that 
$X_\ell(t^2) = \abs\wg^{\lceil(\ell+e-1)/2\rceil}$.  The claimed outcomes 
follow from~\eqref{e:dyadic:sum}.
\end{proof}

\section{Even primes---$m=2$}\label{s:even2}

We may write the anisotropic form as 
$B(x) = a(x_1^2 - \De x_2^2) = \sum_i a_ix_i^2$, where 
$\abs1 \ge \abs{a}, \abs{\De}, \abs{a\De} \ge \abs\wg$ and 
$\De = \eta^2+b$ has quadratic defect $b\oo = \wg^d\oo$.

The Hasse--Minkowski invariant of such a form is 
$(a_1, a_2) = (a, -a\De) = (a,\De)$.  We take $a = 1$ if we wish the invariant 
to be $1$, or use lemma~\thref{l:dyadic:hmi} if we wish it to be $-1$.

Therefore, we have three situations for $\De$: a unit with
quadratic defect $4\oo$, or a unit with quadratic defect $\wg^d\oo$ ($d$ odd
with $0 < d < 2e$), or else $\abs\De = \abs\wg$.  For each situation, we 
further distinguish the cases $(a,\De) = \pm1$.

Before proceeding, we recall \cite[\S63]{Om} that, if $a$ is a unit with
quadratic defect $4\oo$, then 
$\{x^2 - ay^2 : x, y \in k \} = \{t \in k : \ord t \text{ is even}\}$.

\begin{prop}
Let $B(x) = x_1^2 - \De x_2^2$, where $\abs\De = \abs\wg$.  With
$z = q^{-\be}$, $w = zq^{-1}$, and $\abs t = q^{-T}$, we have
\begin{equation*}
  X(\be; t^2)
  = \abs2 \, \frac{1 + w^{2T + e + 1}}{1 - w}.
\end{equation*}
\end{prop}

\begin{proof}
We have
\begin{equation*}
  X_\ell(t^2)
  = \meas\set{x \in \oo^2 : x_1^2 = \De x_2^2 + t^2 \mod 2\wg^\ell}.
\end{equation*}
According to lemma~\thref{l:dyadic:main}, in order to have a solution we 
need $\De x_2^2 = 0 \bmod2\wg^\ell$ or (using the local square theorem and the 
fact that $\ord\De$ is odd) $\De x_2^2 = 0 \bmod 4t^2$.

If $4t^2 = 0 \bmod2\wg^\ell$, we obtain
\begin{equation*}\begin{split}
  X_\ell
  &= \meas\set{x \in \oo^2 : x_1^2 = t^2 \bmod 2\wg^\ell
                 \quad\text{and}\quad
                 \wg x_2^2 = 0 \bmod2\wg^\ell}\\
  &= \abs\wg^{\lceil(\ell+e)/2\rceil} \cdot \abs\wg^{\lceil(\ell+e-1)/2\rceil}
  = \abs\wg^{\ell+e}.
\end{split}\end{equation*}

If $4t^2 \ne 0 \bmod2\wg^\ell$, we obtain
\begin{equation*}\begin{split}
  X_\ell
  &= \meas\set{x \in \oo^2 : x_1^2 = t^2 \bmod 2\wg^\ell
                 \quad\text{and}\quad
                 \wg x_2^2 = 0 \bmod4t^2}\\
  &= 2\abs{\wg^{\ell+e}/2t} \cdot \abs\wg^{e + T}
  = 2 \abs\wg^{\ell+e}.
\end{split}\end{equation*}

Therefore,
\begin{equation*}
  X(\be;t^2)
  = \sum_{\ell \ge 0} z^\ell \abs\wg^{\ell+e}
    + \sum_{\ell \ge 2T+e+1} z^\ell \abs\wg^{\ell+e}
  = \abs2 \, \frac{1 + w^{2T + e + 1}}{1 - w}.
  \qedhere
\end{equation*}
\end{proof}

\begin{prop}
Let $B(x) = a(x_1^2 - \De x_2^2)$, where $\abs\De = \abs\wg$ and $a$ is a unit 
with quadratic defect $4\oo$.  With $z = q^{-\be}$, $w = zq^{-1}$, and 
$\abs t = q^{-T}$, we have
\begin{equation*}
  X(\be; t^2)
  = \abs2\,\frac{1-w^{2T+e+1}}{1-w}.
\end{equation*}
\end{prop}

\begin{proof}
Because $a$ is a unit, $x_1^2 - at^2$ yields exactly the elements of even 
degree, and the quadratic defect of $a$ is $4\oo$, we have, consecutively,
\begin{equation*}\begin{split}
  X_\ell(t^2)
  &= \meas\set{x \in \oo^2 : x_1^2 - at^2 = \De x_2^2 \mod 2\wg^\ell}\\
  &= \meas\set{x \in \oo^2 : \De x_2^2 = 4t^2 = 0 \bmod 2\wg^\ell
                 \quad\text{and}\quad
                 x_1^2 = t^2 \bmod2\wg^\ell}.
\end{split}\end{equation*}
Therefore, $X_\ell(t^2)$ is nonzero only if $0 \le \ell < 2T+e+1$, in which case
\begin{equation*}
  X_\ell(t^2)
  = \abs\wg^{\lceil(\ell+e-1)/2\rceil} \cdot \abs\wg^{\lceil(\ell+e)/2\rceil}
  = \abs\wg^{\ell+e}.
  \qedhere
\end{equation*}
\end{proof}

\begin{prop}
Let $B(x) = \wg(x_1^2 - \De x_2^2)$, where $\De$ is a unit with quadratic defect 
$4\oo$.  Let also $z = q^{-\be}$, $w = zq^{-1}$, and $\abs t=q^{-T}$.

If $2T < e$, then $X(\be; t^2) = 0$.

If $2T \ge e$, write $(T-e)^+ = \max\{T - e, 0\}$ and 
$(T - e)^- = \min\{T - e, 0\}$.  Then $X(\be; t^2)$ is
\begin{equation*}
  \frac{\abs\wg^{\lfloor e/2\rfloor} + z \abs\wg^{\lceil e/2\rceil}
        - z w^e (zw)^{(T - e)^-} (w + 1)}{1 - zw}
  + \frac{w^e (z + w^2)(1 - w^{2(T - e)^+})}{1 - w^2}.
\end{equation*}
\end{prop}

\begin{proof}
Because $x_1^2-\De t^2$ yields exactly the elements of even degree and the 
quadratic defect of $\De$ is $4\oo$, we have, consecutively,
\begin{equation*}\begin{split}
  X_\ell(t^2)
  &= \meas\set{x \in \oo^2 :
                 \wg(x_1^2 - \De x_2^2) = t^2 = 0 \mod2\wg^\ell}\\
  &= \meas\set{x \in \oo^2 : t^2 = \wg4 x_2^2 = 0 \bmod2\wg^\ell
                 \quad\text{and}\quad
                 \wg x_1^2 = \wg x_2^2 \bmod 2\wg^\ell}.
\end{split}\end{equation*}
Considering only $0 \le \ell < 2T - e + 1$, we obtain
\begin{equation*}
  X_\ell(t^2)
  = \abs\wg^{\lceil\max\{0,\ell-e-1\}/2\rceil}
    \cdot \abs\wg^{\lceil(\ell+e-1)/2\rceil}.
\end{equation*}

If $2T < e$, then $2T - e + 1 \le 0$ always and $X(\be; t^2) = 0$.

If $e \le 2T < 2e + 1$, then $0 \le \ell < 2T - e + 1 \le e + 1$ always, and
by \eqref{e:dyadic:sum},
\begin{equation*}\begin{split}
  X(\be; t^2)
  &= \sum_{0 \le \ell < 2T - e + 1} z^\ell \abs\wg^{\lceil(\ell+e-1)/2\rceil}\\
  &= \abs\wg^{\lfloor e/2\rfloor} 
     \frac{1 - (zw)^{\lceil(2T-e+1)/2\rceil}}{1 - zw}
     + z \abs\wg^{\lceil e/2\rceil}
       \frac{1 - (zw)^{\lfloor(2T-e+1)/2\rfloor}}{1 - zw}.
\end{split}\end{equation*}

If $2T \ge 2e + 2$, then
\begin{equation*}
  X(\be; t^2)
  = \sum_{0\le\ell<e+1} z^\ell \abs\wg^{\lceil(\ell+e-1)/2\rceil}
    + \sum_{e+1\le\ell<2T-e+1} z^\ell \abs\wg^{2\lceil(\ell-e-1)/2\rceil + e}.
\end{equation*}
The first sum is the same as before, but with $T$ replaced by $e$.  The second
sum is obtained from \eqref{e:dyadic:sum} too (note we use 
$w^2 = z^2 \abs\wg^2$ instead of $zw = z^2 \abs\wg$):
\begin{equation*}
  z w^e \sum_{0 \le \ell < 2T - 2e} z^\ell \abs\wg^{2\lceil\ell/2\rceil}
  = \frac{(zw^e + w^{e+2})(1 - w^{2T - 2e})}{1 - w^2}.
  \qedhere
\end{equation*}
\end{proof}

\begin{prop}\label{p:ex:n:4}
Let $B(x) = a(x_1^2 - \De x_2^2)$, where $\De = 1 + \wg^d v$ is a unit with
quadratic defect $\wg^d\oo$, $d$ is odd, $a = 1 + \wg u$ is a unit with 
quadratic defect $\wg\oo$, and $(a, \De) = -1$.  Let also $z = q^{-\be}$,
$w = zq^{-1}$, and $\abs t = q^{-T}$.

If $d = 1$ and $e > 1$, then
\begin{equation*}
  X(\be; t^2)
  = \abs2 \, \frac{1 - w^{2T + 2 \lceil(e+1)/2\rceil - e}}{1 - w}.
\end{equation*}

If $2T + 2 \ge e + 1 \ge d$ (with $d > 1$ or $e = 1$), then
\begin{equation*}
  X(\be; t^2)
  = \abs2 \, \abs\wg^{(1-d)/2} \, \frac{1 - w^{2T+2-e}}{1 - w}.
\end{equation*}
If $2T + 2 \ge e + 1$ and $d > e + 1$, then
\begin{equation*}
  X(\be; t^2)
  = \frac{\abs\wg^{\lceil e/2\rceil} + w \abs\wg^{\lfloor e/2\rfloor}}{1 - zw}
    - \frac{z^{d-e} \abs\wg^{(d+1)/2} (z - w)}{(1 - w)(1 - zw)}
    - \frac{w^{2T + 2 - e} \abs\wg^{e+(1-d)/2}}{1 - w}.
\end{equation*}
If $2T + 2 \le e$ (with $d > 1$), then $X(\be; t^2) = 0$.
\end{prop}

\begin{proof}
We have
\begin{equation*}
  X_\ell(t^2)
  = \meas\set{x \in \oo^2 : x_1^2 = \De x_2^2 + at^2 \mod 2\wg^\ell}.
\end{equation*}

If $\abs t > \abs{x_2}$, we need $\wg t^2 = 0 \bmod2\wg^\ell$.

If $\abs t < \abs{x_2}$, we need $\wg t^2 = \wg^d x_2^2 = 0 \bmod 2\wg^\ell$.

If $\abs t = \abs{x_2}$, we write $x_2 = tz$ (with $z$ a unit) and observe
\begin{equation*}
  \De x_2^2 + at^2
  = t^2((1+z)^2 + \wg u - 2z + \wg^d v z^2).
\end{equation*}
Therefore, if $d > 1$ we require $\wg t^2 = 0\bmod2\wg^\ell$.
If $d = e = 1$ and $\wg = 2$, and because $(1+2u, 1+2v) = -1$, we always have
$u - z + v z^2 \ne 0 \bmod 2$, so we require $2t^2 = 0 \bmod 2\wg^\ell$.
Finally, if $d = 1$ and $e > 1$, then we can choose $v = u$ and write
\begin{equation*}
  \De x_2^2 + at^2
  = t^2 (1 + \wg u)(1 + z^2)
  = t^2 \bigl((1 + \wg u)(1 + z)^2 - 2(1 + \wg u) z\bigr);
\end{equation*}
we thus require $t^2\wg(1 + z)^2 = 0 \bmod2\wg^\ell$ and either
$2\wg t^2 = 0 \bmod2\wg^\ell$ (if $e$ is even) or $2t^2 = 0 \bmod 2\wg^\ell$
(if $e$ is odd).

If $d = 1$ and $e > 1$, we required $\wg t^2 = \wg x_2^2 = 0 \bmod 2\wg^\ell$ 
(when $\abs t \ne \abs{x_2}$) or $\wg(t+x_2)^2 = 0 \bmod 2\wg^\ell$ and
$\wg^{2\lceil(e+1)/2\rceil - e - 1} t^2 = 0 \bmod \wg^\ell$
(when $\abs t = \abs{x_2}$).  We obtain
\begin{equation*}
  X_\ell(t^2)
  = \abs\wg^{\lceil(\ell+e)/2\rceil} \cdot
    \abs\wg^{\lceil(\ell+e-1)/2\rceil}
  = \abs\wg^{\ell+e}
\end{equation*}
if $\ell < 2T + 2\lceil(e+1)/2\rceil - e$ and $X_\ell(t^2) = 0$ otherwise.

If $d > 1$ or $e = 1$, we required $\wg t^2 = \wg^d x_2^2 = 0\bmod 2\wg^\ell$.
Therefore,
\begin{equation*}
  X_\ell(t^2)
  = \abs\wg^{\lceil(\ell+e)/2\rceil} \cdot
    \abs\wg^{\max\{0, \lceil(\ell+e-d)/2\rceil\}}
\end{equation*}
if $\ell < 2T + 2 - e$ and $X_\ell(t^2) = 0$ otherwise.
\end{proof}

So far, we have relied on the first method discussed in section~\ref{s:dyadic}.
For all remaining quadratic forms, we will use the second one.  In particular, 
all that follows is valid only for the unramified case.

The strategy is always the same: we first reduce the equation modulo $2$.
This corresponds to $\ell = 0$ and suggests a substitution for one of the
variables.  That variable will be set modulo $2$---hence, we always have an 
extra factor $q^{-1}$ in the final calculation of $X_\ell(t^2)$.

Applying the substitution and simplifying, we obtain a new equation, modulo 
$2^\ell$ (the original equation was modulo $2^{\ell+1}$).  At this point, we
consider the case $\ell = 1$.  If the equation thus reduced is linear with 
unit coefficient, we know how many solutions it has.  If the equation is
quadratic, we apply lemma~\thref{l:dyadic:proj}: either we obtain new
conditions on other variables, typically allowing us to divide the original
equation by $4$ and conclude $X_\ell(t^2) = q^{-m} X_{\ell-2}(t^2/4)$,
or we obtain a solution count.

\begin{prop}
Let $B(x) = x_1^2 - \De x_2^2$, where $\De = 1 + 4v$ is a unit with quadratic 
defect $4\oo$ and $v$ is a unit.  In the unramified case, with $z = q^{-\be}$, 
$w = zq^{-1}$, and $\abs t = q^{-T}$, we have
\begin{equation*}
  X(\be; 1)
  = \frac{\abs 2}{1 - w}
  \qquad\text{and}\qquad
  X(\be; 4t^2)
  = \abs2 \, \frac{1 + z + w^{2T} (zw + w^3)}{1 - w^2}.
\end{equation*}
\end{prop}

\begin{proof}
The equation is $x_1^2 - \De x_2^2 = t^2 \mod 2^{\ell+1}$.  Considering 
$\ell = 0$, we are led to $x_1 = x_2 + t + 2b$, for some $b \in \oo$.  We 
substitute and simplify:
\begin{equation*}
  2b^2 + x_2 t + 2x_2 b + 2tb - 2v x_2^2 = 0 \mod 2^\ell.
\end{equation*}

If $t = 1$ and $\ell \ge 1$, we clearly can obtain a unique 
$x_2 \bmod 2^\ell$.  Therefore, recalling $x_1$ was set modulo $2$, we have 
$X_\ell(1) = q^{-1-\ell}$.

If $2 \mid t$, the equation holds for $\ell \le 1$; that is, 
$X_0(t^2) = X_1(t^2) = q^{-1}$.  If $\ell \ge 2$, we divide further:
\begin{equation*}
  b^2 + x_2 \tfrac t2 + x_2 b + tb - v x_2^2 = 0 \mod 2^{\ell-1}.
\end{equation*}
We note that $(1+2v, -1) = -1$, so we may apply lemma~\thref{l:dyadic:proj}.

For $t = 2$, the lemma tells us the measure of the solution set with 
respect to $b$ and $x_2$ is $q^{-\ell+1} + q^{-\ell}$, so with respect to
$x_1$ and $x_2$, for $\ell \ge 2$, we have
\begin{equation*}
  X_\ell(4)
  = q^{-\ell} + q^{-\ell-1}.
\end{equation*}

If $4 \mid t$ and $\ell \ge 2$, then $x_2 = b = x_1 = 0 \bmod 2$, so 
$X_\ell(t^2) = q^{-2} X_{\ell-2}(t^2/4)$.
\end{proof}

\begin{prop}\label{p:ex:n:8}
Let $B(x) = x_1^2 - \De x_2^2$, where $\De = 1 + 2 v$ is a unit with 
quadratic defect $2\oo$ and $v$ is a unit.  In the unramified case, with 
$z = q^{-\be}$, $w = zq^{-1}$, and $\abs t = q^{-T}$, we have
\begin{equation*}
  X(\be; t^2)
  = \abs2 \, \frac{1 + w^{2T+1}}{1 - w}.
\end{equation*}
\end{prop}

\begin{proof}
The equation is $x_1^2 - \De x_2^2 = t^2 \bmod 2^{\ell+1}$.  Replacing
$x_1 = x_2 + t + 2b$ and simplifying:
\begin{equation*}
  2b^2 + x_2 t + 2x_2b + 2tb - v x_2^2 = 0 \mod 2^\ell.
\end{equation*}

If $t = 1$, we have $x_2 = 0 \bmod2$ or $x_2 = v^{-1} \bmod 2$, and, because 
the coefficient of $x_2$ is a unit, solutions can be refined modulo $2^\ell$.
Therefore, $X_0(1) = q^{-1}$ and 
$X_\ell(1) = q^{-1} \cdot 2q^{-\ell} = 2q^{-\ell-1}$ for $\ell \ge 1$.

If $2 \mid t$, then $x_2 = x_1 = 0 \bmod 2$.  Therefore, 
$X_0(t^2) = q^{-1}$, $X_1(t^2) = q^{-2}$, and 
$X_\ell(t^2) = q^{-2} X_{\ell-2}(t^2/4)$ for $\ell \ge 2$.
\end{proof}

\section{Even primes---$m=3$}\label{s:even3}

A ternary quadratic form with discriminant $\De$ is anisotropic if and only if 
its Hasse--Minkowski invariant is $-(-1, \De)$.

The form $B(x) = x_1^2 - a(x_2^2 - \De x_3^2)$ has discriminant $\De$ and
Hasse--Minkowski invariant $(-a, a\De) = (-a, \De) = (-1, \De)(a, \De)$.
Therefore, it is anisotropic when $(a, \De) = -1$.  If $\abs\De = \abs\wg$, we 
may take any $a$ with quadratic defect $4\oo$.  If $\De$ is a 
unit with quadratic defect $\wg^d\oo$ (with $d$ odd), lemma~\thref{l:dyadic:hmi} 
yields a unit $a$ with quadratic defect $\wg\oo$ and $(a, \De)=-1$.

If $\De$ is a unit square or a unit with quadratic defect $4\oo$, the form is
anisotropic if and only if the Hasse--Minkowski invariant is $-1$.  We choose
$a = 1 + \wg u$, where $u$ is a unit and $(a, -1) = -1$, provided by 
lemma~\thref{l:dyadic:hmi}.  The form $B(x) = a(x_1^2 + x_2^2) - \De x_3^2$
has discriminant $\De$ and Hasse--Minkowski invariant $(a, a) = (a, -1) = -1$.

We are still using the second method, so we consider only the unramified case.

\begin{prop}
Let $B(x) = x_1^2 - a(x_2^2 - 2 v x_3^2)$, where $a = 1 + 4u$ is a unit with
quadratic defect $4\oo$, and $u$ and $v$ are units.  In the unramified case, 
with $z = q^{-\be}$, $w = zq^{-1}$, and $\abs t = q^{-T}$, we have
\begin{equation*}
  X(\be; 1)
  = \frac{\abs 2}{1 - w}
\end{equation*}
and
\begin{equation*}
  X(\be; t^2)
  = \abs 2 \frac{1 + w}{1 - w^2q^{-1}} 
    + w^{2T} q^{-T} \frac{1 - w^2q^{-2}}{(1 - w)(1 - w^2q^{-1})},
  \qquad\text{if $T \ge 1$}.
\end{equation*}
\end{prop}

\begin{proof}
The equation is $x_1^2 - a(x_2^2 - 2 v x_3^2) = t^2 \bmod 2^{\ell+1}$.
We substitute $x_1 = x_2 + t + 2b$, simplify, and obtain
\begin{equation*}
  2b^2 - x_2t - 2x_2b - 2tb - 2ux_2^2 + a v x_3^2 = 0 \mod 2^\ell.
\end{equation*}

If $t = 1$ and $x_3$ is fixed, we obtain exactly one solution 
$x_2 \bmod 2^\ell$.  Therefore, 
$X_\ell(1) = q^{-1} \cdot q^{-\ell} = q^{-\ell - 1}$.

If $2 \mid t$ and $\ell \ge 1$, we are led to $x_3 = 0 \bmod 2$.   
For such $t$, we have $X_0(t^2) = q^{-1}$ and $X_1(t^2) = q^{-2}$.  If 
$\ell \ge 2$, we have also
\begin{equation*}
  b^2 - x_2 \tfrac t2 - x_2 b - 2 \tfrac t2 b - ux_2^2 + 2 av \tfrac{x_3^2}4
  = 0 \mod 2^{\ell-1}.
\end{equation*}

If $t = 2$, we obtain $q^{-\ell+1} + q^{-\ell}$ solutions with respect to
$b$ and $x_2$, or $q^{-\ell} + q^{-\ell-1}$ with respect to $x_1$ and $x_2$.
Recalling that $x_3 = 0 \bmod 2$, we see that, for $\ell \ge 2$,
$X_\ell(4) = q^{-\ell-1} + q^{-\ell-2}$.

If $4 \mid t$, then we obtain $x_2 = 0 \bmod 2$, so 
$x_1 = x_2 = x_3 = t = 0 \bmod 2$, and we have, for $\ell \ge 2$,
$X_\ell(t^2) = q^{-3} \, X_{\ell-2}(t^2/4)$.
\end{proof}

\begin{prop}\label{p:ex:n:9}
Let $B(x) = x_1^2 - a(x_2^2 - \De x_3^2)$, where $a = 1 + 2 u$ and $\De$ are 
units with quadratic defect $2\oo$, $(a, \De) = -1$, $-a\De = 1 + 2v$, and $u$
and $v$ are units.  In the unramified case, with $z = q^{-\be}$, 
$w = zq^{-1}$, and $\abs t = q^{-T}$,  we have
\begin{equation*}
  X(\be; t^2)
  = \abs 2 \, \frac{1 + wq^{-1}}{1 - w^2q^{-1}}
    + \abs 2 \, w^{2T+1} q^{-T} \frac{1 - w^2q^{-2}}{(1 - w)(1 - w^2q^{-1})}.
\end{equation*}
\end{prop}

\begin{proof}
The equation is $x_1^2 - ax_2^2 + a\De x_3^2 = t^2 \bmod 2^{\ell+1}$.
Replacing $x_1 = x_2 + x_3 + t + 2b$ and simplifying, we obtain
\begin{equation*}
  2b^2 + x_2 x_3 + x_2 t + 2x_2 b + x_3 t + 2x_3b
  + 2tb -u x_2^2 - v x_3^2 = 0 \bmod 2^\ell.
\end{equation*}
In particular, $X_0(t^2) = q^{-1}$.

If $t = 1$ and $\ell \ge 1$, and recalling we already set $x_1 \bmod 2$, we 
obtain $X_\ell(1) = q^{-\ell-1} + q^{-\ell-2}$.

If $2 \mid t$ and $\ell \ge 1$, we conclude $x_2 = x_3 = 0 \bmod 2$, so 
$X_1(t^2) = q^{-3}$ and, for $\ell \ge 2$, 
$X_\ell(t^2) = q^{-3} \, X_{\ell-2}(t^2/4)$.
\end{proof}

\begin{prop}\label{p:ex:n:3}
Let $B(x) = a(x_1^2 + x_2^2) - x_3^2$, where $a = 1 + 2u$ is a unit with 
quadratic defect $2\oo$, $(a, -1) = -1$, and $u$ is a unit.  In the unramified
case, with $z = q^{-\be}$, $w = zq^{-1}$, and $\abs t = q^{-T}$, we have
\begin{equation*}
  X(\be; t^2)
  = \abs 2 \, \frac{(1 + wq^{-1})(1 - w^{2T+2} q^{-T-1})}{1 - w^2q^{-1}}.
\end{equation*}
\end{prop}

\begin{proof}
The equation is $x_1^2 + x_2^2 - a \De x_3^2 = a t^2 \bmod 2^{\ell+1}$.
Replacing $x_1 = x_2 + x_3 + t + 2b$ and simplifying, we obtain
\begin{equation*}
  x_2^2 + 2b^2 + x_2x_3 + x_2t + 2x_2b + x_3t + 2x_3b + 2tb - u x_3^2 - ut^2
  = 0 \mod 2^\ell.
\end{equation*}
In particular, $X_0(t^2) = q^{-1}$.

If $t = 1$ and $\ell = 1$, lemma~\thref{l:dyadic:proj} tells us 
$x_2 = x_3 = 1 \bmod2$ and $X_1(1) = q^{-3}$.

If $t = 1$ and $\ell = 2$, taking into account that $x_2 = x_3 = 1 \bmod2$,
we see the equation is equivalent to
\begin{equation*}
  b^2 -b - u = 0 \mod 2.
\end{equation*}
This equation has no solution, therefore, $X_\ell(1) = 0$ for $\ell \ge 2$.

If $2 \mid t$ and $\ell \ge 1$, then the equation leads to
$x_2 = x_3 = t = x_1 = 0 \bmod 2$.  This means $X_1(t^2) = q^{-3}$
and $X_\ell(t^2) = q^{-3} X_{\ell-2}(t^2/4)$ if $\ell \ge 2$.
\end{proof}

\begin{prop}
Let $B(x) = a(x_1^2 + x_2^2) - \De x_3^2$, where $\De = 1 + 4u$ is a unit with
quadratic defect $4\oo$, $a = 1 + 2u$ is a unit with quadratic defect $2\oo$,
$u$ is a unit, and $(a, -1) = -1$.  In the unramified case, with 
$z = q^{-\be}$, $w = zq^{-1}$, and $\abs t = q^{-T}$, we have
\begin{equation*}
  X(\be; t^2)
  = \abs2 \, \frac{1 + wq^{-1}}{1 - w^2q^{-1}}
    + \abs 2 \, w^{2T+2}q^{-T-1} \,
      \frac{(1 + w)(1 - wq^{-1})}{(1 - w)(1 - w^2q^{-1})}.
\end{equation*}
\end{prop}

\begin{proof}
The equation is $x_1^2 + x_2^2 - a \De x_3^2 = a t^2 \mod 2^{\ell+1}$.
Replacing $x_1 = x_2 + x_3 + t + 2b$ and simplifying, we obtain
\begin{equation*}
  x_2^2 + 2b^2 + x_2x_3 + x_2t + 2x_2b + x_3t + 2x_3b + 2tb
  - u(3 + 4u) x_3^2 - ut^2 = 0 \mod 2^\ell.
\end{equation*}
Clearly, $X_0(t^2) = q^{-1}$.

If $t = 1$, the equation, reduced modulo $4$, is equivalent to
\begin{equation*}
  (x_2 + 1)^2 + (x_2 + 1)(x_3 - 1) + 2b(b + x_2 + x_3 + 1)
  + u (x_3 - 1)^2 + 2 u(x_3 - 1) = 0 \mod 4.
\end{equation*}
If $\ell \ge 1$, we must have $x_2 = x_3 = 1 \bmod2$, so $X_1(1) = q^{-3}$.
If $\ell \ge 2$, we must also have $b^2 = b \bmod2$, that is, $b = 0 \bmod2$
or $b = 1 \bmod2$, which leads to $X_2(1) = 2q^{-4}$.  Therefore,
$X_\ell(1) = 2q^{-\ell-2}$ for $\ell\ge 2$.

If $2 \mid t$, the original equation yields $x_2 = x_3 = t = x_1 = 0 \bmod2$.  
Therefore, $X_1(t^2) = q^{-3}$ and $X_\ell(t^2) = q^{-3} X_{\ell-2}(t^2/4)$ if 
$\ell \ge 2$.
\end{proof}

\section{Even primes---$m=4$}\label{s:even4}

In the $m = 4$ case, we have only one equivalence class of anisotropic forms.

\begin{prop}\label{p:ex:n:10}
Let $B(x) = x_1^2 + x_2^2 - a(x_3^2 + x_4^2)$, where $a = 1 + 2u$, $u$ is a 
unit, and $(a, -1) = -1$.  In the unramified case, with $z = q^{-\be}$, 
$w = zq^{-1}$, and $\abs t = q^{-T}$, we have
\begin{equation*}
  X(\be; 1)
  = \frac{\abs2}{1 - w}
\end{equation*}
and
\begin{equation*}
  X(\be; t^2)
  = \frac{\abs 2}{1 - wq^{-1}}
    + w^{2T} q^{-2T} \frac{1 - wq^{-2}}{(1 - w)(1 - wq^{-1})},
  \qquad\text{if $T \ge 1$}.
\end{equation*}
\end{prop}

\begin{proof}
We know $x_1 + x_2 + x_3 + x_4 + t = 0 \bmod 2$.  Replacing
$x_1 = Z + t$ and $x_2 = x_3 + x_4 + Z + 2b$ into 
$B(x) = t^2 \bmod 2^{\ell+1}$ and simplifying, we obtain
\begin{equation*}
  -ux_3^2 + x_3x_4 - ux_4^2 + (Z+4b)x_3 + (Z+4b)x_4 + Z^2 + Zt + 4bZ + 4b^2
  = 0 \mod 2^\ell.
\end{equation*}
As usual, $X_0(t^2) = q^{-1}$.

If $t = 1$, we only need to reduce the equation modulo $4$, we obtain
\begin{equation}\label{e:even4:}
  -ux_3^2 + x_3x_4 - ux_4^2 + Zx_3 + Zx_4 + Z^2 + Zt
  = 0 \mod 4.
\end{equation}
Replacing $x_3 = x_4 = Z \bmod 2$, we obtain $Zt \bmod2$.

If $t = 1$, $\ell = 1$, and $Z$ is a unit, then the solution set (with 
respect to $x_3$ and $x_4$, and modulo $2$) has measure $q^{-1} + q^{-2}$.  If
$Z$ is not a unit, then $x_3 = x_4 = 0 \bmod 2$.  Therefore (we must not forget
that $x_2$ was set modulo $2$),
\begin{equation*}
  X_1(1)
  = q^{-1} (1 - q^{-1})(q^{-1} + q^{-2}) + q^{-4}
  = q^{-2}.
\end{equation*}

If $t = 1$ and $\ell = 2$, and $Z$ is a unit, then the solution set (with 
respect to $x_3$ and $x_4$, and modulo $4$) has measure $q^{-2} + q^{-3}$.  If
$Z$ is not a unit, then $x_3 = x_4 = 0 \bmod 2$ and $Z = 0 \bmod 4$.  Therefore 
(again, $x_2$ was set modulo $2$),
\begin{equation*}
  X_2(1)
  = q^{-1} (1 - q^{-1})(q^{-2} + q^{-3}) + q^{-5}
  = q^{-3}.
\end{equation*}

More generally, $X_\ell(1) = q^{-\ell-1}$.

If $2 \mid t$, we return to equation \eqref{e:even4:}.  
Lemma~\thref{l:dyadic:proj} tells us $x_3 = x_4 = Z \bmod2$, 
so $X_1(t^2) = q^{-3}$ (again, we must not forget $x_2$ was set modulo $2$).  
For $\ell \ge 2$, we substitute $x_3 = Z + 2X$ and $x_4 = Z + 2Y$ and simplify:
\begin{equation*}
  -uZ^2 + Z \tfrac t2
  = 0 \mod 2.
\end{equation*}

If $t = 2$ and $Z = 0 \bmod 2$, we conclude 
$x_1 = x_2 = x_3 = x_4 = t = 0 \bmod2$, so the contribution (in the equation
modulo $2^\ell$) is $q^{-4} X_{\ell-2}(1) = q^{-3-\ell}$.  If $t = 2$ and 
$Z = u^{-1} \bmod2$, then $Z$ is determined modulo $2^{\ell-1}$ (recall we
divided by $2$ in most recent simplification), so the contribution is 
$q^{-\ell+1}q^{-1}q^{-1}q^{-1} = q^{-2-\ell}$.  Therefore, 
$X_\ell(4) = q^{-2-\ell} + q^{-3-\ell}$.

If $4 \mid t$ (still for $\ell \ge 2$), we may also conclude $Z = 0 \bmod2$,
so $X_\ell(t^2) = q^{-4} X_{\ell-2}(t^2/4)$.
\end{proof}

\section{Some examples in $k = \QQ$}\label{s:examples}

In order to determine the global period in all cases, we still need the local
factors at some even primes (the ramified cases we could not address here), the
missing normalization constant in proposition \thref{p:odd:lpi}, and
the local factors at archimedean primes.

Therefore, for these examples we will ignore multiplicative constants and 
consider only $k = \QQ$ and the standard form 
$\sum_{i=1}^{n+2} x_i^2 - x_{n+3}^2$ with signature $(n+2, 1)$ on
$k^{n+2}\oplus k \cdot e_-$.

At the archimedean place, $k_v = \RR$ and the restriction of the form 
to $\RR^{n+2}$ is anisotropic.  Hence, as mentioned in the introduction, the 
local factor of the period is simply 
$\vol(\Sv\lqu\Hv) = \vol(\gO(n+1, \RR)\lqu\gO(n+2, \RR))$, a multiplicative
constant.

At non-archimedean places, we have simply $q_v = p$ (where $p$ is the
prime).  As the discriminant is $\De = \pm1$, there are no bad odd primes.

If $\De = 1$, the associated quadratic character is the trivial character 
$\chi_0$ (its 
$L$--function is the Riemann zeta function).  If $\De = -1$, the associated
quadratic character is the character $\chi_1$ given by $\chi_1(p) = 1$ if 
$p = 1 \bmod 4$, $\chi_1(p) = -1$ if $p = 3 \bmod 4$, or $\chi_1(p) = 0$ if 
$p$ is even.  Also,
\begin{equation*}
  \ze(s)
  = \prod_p \frac1{1 - p^s}
  \qquad\text{and}\qquad
  L(s, \chi)
  = \prod_p \frac1{1 - \chi(p)\, p^s}.
\end{equation*}

We will also limit our attention to periods $(E_\ph, 1)_H$, that is, periods
of the Eisenstein series alone, rather than against a cuspidal $F$, in
which case the local parameters are $\be = 0$.

With $\al = (n+1) s$, we saw before \cite{Bo} that, up to multiplicative 
constants and correction factors at $p = 2$ (determined in this paper), the
global period is
\begin{align}\label{e:g-per-odd}
  (E_\ph, 1)_H
  &= \frac{\ze(\al - n)}
      {L(\al - \lfloor\frac n2\rfloor, \chi)},
      && \text{if $n$ is odd;}\\\label{e:g-per-even}
  (E_\ph, 1)_H
  &= \frac{\ze(\al - n)}
      {\ze(2\al - n) / L(\al - \frac n2, \chi)},
      && \text{if $n$ is even}
\end{align}
(where $\chi = \chi_0$ when $\De = 1$, and $\chi = \chi_1$ when $\De = -1$).

Say $B^{n+2}$ is the original form with signature $(n+2,0)$ and $B^n$ is the 
original form $B$ (from the discussion, in section~\ref{s:setup}, of the 
measure on $\Th\lqu H$).  Say also that $B^{n-2k}$ is the form obtained after 
taking $k$ hyperbolic planes away (for the dimension reduction cited at the 
end of section \ref{s:setup}), until we get an anisotropic form $B^m$, with 
$n = m + 2k$---those are the forms whose $X$ functions we computed in sections 
\ref{s:even0} through \ref{s:even4}.

All those forms have the same discriminant $\De$.  However, their 
Hasse--Minskowski invariants are not the same.  Let $\hmi B$ denote the
Hasse--Minkowski invariant of a form $B$.  In general \cite{Ca,Om}, 
if $B$ is the sum of two forms $C$ and $D$, then 
$\hmi B = (\det C, \det D)\cdot\hmi C\cdot\hmi D$.

A hyperbolic plane has determinant $-1$ and invariant $(1, -1) = 1$.  Also,
$B^{n+2}$ has invariant $1$ and determinant $1$, so 
$\det B^{n-2k} = (-1)^{k+1}$ and $(-1, \det B^{n-2k}) = (-1)^{k+1}$.  This
means that with even $k$ we change the sign of the invariant, and with odd $k$ 
we keep it.  That is, starting with $\hmi B^{n+2}$, and taking one hyperbolic 
plane away at a time, we obtain $1$, $-1$, $-1$, $1$, and then repeat with 
period four.

Applying the discussions at the beginning of sections \ref{s:even2} through
\ref{s:even4} to our current case, we see that $B^2$ is anisotropic if and 
only if $\De = -1$, that $B^3$ is anisotropic if and only if $\hmi B^3 = -\De$,
and that $B^4$ is anisotropic if and only if $\De = \hmi B^4 = 1$.

We thus obtain the information in table~\ref{t:X}, for $n \ge 3$.  With 
$n < 3$, $H$ would be anisotropic at $p = 2$, the local factor would be a 
constant, and the results in this paper would not be used.  Taking our choice 
$\be = 0$ and the dimension reduction at the end of section~\ref{s:setup}
into account, we use $z = q^{-k}$ and $w = q^{-k-1}$ and abbreviate 
$u = q^{-n}$ (this is always what is raised to the power $T$).  Additionally,
as the local factor is obtained by integration of this $X$ with respect to $t$
(that is, in terms of these formulas, a summation with respect to $T = \ord t$)
and we are missing a multiplicative constant, we multiply by a common factor
so that the result is as close as possible to the form $1 - A u^T$.

\begin{table}
\caption{Taking $k$ hyperbolic planes from $B^n$, we obtain the anisotropic 
form $B^m$, whose $X$ function with list, as well as the respective 
proposition.}\label{t:X}

\begin{tabular}{ccccccc}\toprule
$n$         & $\De$ & $k$         & $m$ & $\hmi B^m$ & prop. &
$X^m(\be; t^2)$
\\\midrule
$3 + 8\ell$  & $ 1$ & $4\ell$     & $3$ & $-1$ & \eqref{p:ex:n:3} &
$1 - u u^T$;
\\\addlinespace
$4 + 8\ell$  & $-1$ & $4\ell + 1$ & $2$ & $-1$ & \eqref{p:ex:n:4} &
$1 - w u^T$;
\\\addlinespace
$5 + 8\ell$  & $-1$ & $4\ell + 2$ & $1$ & $ 1$ & \eqref{p:ex:n:5} &
$1 - \dfrac{u + z}{1 + z} u^T$;
\\\addlinespace
$6 + 8\ell$  & $ 1$ & $4\ell + 3$ & $0$ & $ 1$ & \eqref{p:ex:n:6} &
$\begin{cases}
  1, &\text{if $T=0$},\\
  0, &\text{otherwise};
\end{cases}$
\\\addlinespace
$7 + 8\ell$  & $ 1$ & $4\ell + 3$ & $1$ & $ 1$ & \eqref{p:ex:n:7} &
$1 + \dfrac{1 - w - u}{1 + w - u} u u^T$;
\\\addlinespace
$8 + 8\ell$  & $-1$ & $4\ell + 3$ & $2$ & $ 1$ & \eqref{p:ex:n:8} &
$1 + w u^T$;
\\\addlinespace
$9 + 8\ell$  & $-1$ & $4\ell + 3$ & $3$ & $ 1$ & \eqref{p:ex:n:9} &
$1 - \dfrac{u - w}{1 - w} u^T$;
\\\addlinespace
$10 + 8\ell$ & $ 1$ & $4\ell + 3$ & $4$ & $ 1$ & \eqref{p:ex:n:10} &
$\begin{cases}
  \dfrac{1 - wq^{-1}}{1 - w}, &\text{if $T = 0$},\\
  1 + \dfrac{1 - wq^{-2}}{q^{-1}(1 - w)}  u^T, &\text{otherwise}.
\end{cases}$
\\\addlinespace\bottomrule
\end{tabular}
\end{table}

That choice simplifies substantially the computation of $\Pi$.  Indeed, with
$a = q^{-\al}$, the definition of $\Pi$ (from section \ref{s:setup}, using 
the multiplicative measure) is
\begin{equation*}
  \Pi(\al, \be)
  = \int_{\kx\cap\oo} \abs t^\al X(\be; t^2) \dd t
  = \sum_{T \ge 0} X(\be; t^2) a^T.
\end{equation*}
If, up to the common factors we dropped, $X = 1 - A u^T$, this becomes
\begin{equation*}
  \Pi(\al, \be)
  = \sum_{T \ge 0} \bigl(a^T - A (au)^T\bigr)
  = \frac1{1 - a} - \frac A{1 - au}
  = \frac{(1 - A) - a (u - A)}{(1 - a)(1 - au)}.
\end{equation*}
This becomes even simpler when $A = \frac{u - v}{1 - v}$ for some $v$, as in 
that case we obtain
\begin{equation*}
  \Pi(\al, \be)
  = \frac{(1 - u)(1 - av)}{(1 - a)(1 - au)(1 - v)}
  \stackrel{\substack{\text{up to}\\\text{constant}}\vphantom{\Bigr)}} {=}
  \frac{1 - av}{(1 - a)(1 - au)}.
\end{equation*}
Table~\ref{t:Pi} summarizes the results.

\begin{table}
\caption{Taking $k$ hyperbolic planes from $B^n$, we obtain an anisotropic 
form $B^m$.  With $z = q^{-k}$, $w = q^{-k-1}$, and 
$u = q^{-n} = q^{-m-2k} = z^2q^{-m}$, we obtain $\Pi^m(al, k)$ from 
$X^m(k; t^2)$.  When relevant, we indicate the choice of $v$ that permits the 
simplification discussed in the 
text.  We drop any common factors that do not depend on $a$.}\label{t:Pi}

\begin{tabular}{cccc@{}c@{}}\toprule
$n$          & $k$         & $m$ & $v$ & $\Pi^m(\al, k)$
\\\midrule
$3 + 8\ell$  & $4\ell$     & $3$ &  &
$\dfrac1{(1 - a)(1 - au)} = Z(\al) \, Z(\al + n)$;
\\\addlinespace
$4 + 8\ell$  & $4\ell + 1$ & $2$ & $-w$ &
$\dfrac{1 + aw}{(1 - a)(1 - au)}
 = \dfrac{Z(\al) \, Z(\al + n) \, Z(\al + \frac n2)}{Z(2\al + n)}$;
\\\addlinespace
$5 + 8\ell$  & $4\ell + 2$ & $1$ & $-z$ &
$\dfrac{1 + az}{(1 - a)(1 - au)}
 = \dfrac{Z(\al) \, Z(\al + n) \, Z(\al + k)}{Z(2\al + n - 1)}$;
\\\addlinespace
$6 + 8\ell$  & $4\ell + 3$ & $0$ & &
$1$;
\\\addlinespace
$7 + 8\ell$  & $4\ell + 3$ & $1$ & $\frac{2u}{1+w+u}$ &
$\dfrac{1 - av}{(1 - a)(1 - au)}
 = \dfrac{Z(\al) \, Z(\al + n)}{Z(\al - \log_q v)}$;
\\\addlinespace
$8 + 8\ell$  & $4\ell + 3$ & $2$ & $w$ &
$\dfrac{1 - aw}{(1 - a)(1 - au)}
 = \dfrac{Z(\al) \, Z(\al + n)}{Z(\al + k + 1)}$;
\\\addlinespace
$9 + 8\ell$  & $4\ell + 3$ & $3$ & $w$ &
$\dfrac{1 - aw}{(1 - a)(1 - au)}
 = \dfrac{Z(\al) \, Z(\al + n)}{Z(\al + k + 1)}$;
\\\addlinespace
$10 + 8\ell$ & $4\ell + 3$ & $4$ & &
$\dfrac{(1 - aw)(1 + awq^{-1})}{(1 - a)(1 - au)}
 = \dfrac{Z(\al) \, Z(\al + n) \, Z(\al + \frac n2)}
         {Z(\al + k + 1) \, Z(2\al + n)}$.
\\\addlinespace\bottomrule
\end{tabular}
\end{table}

Recall now that in proposition \thref{p:odd:lpi} we identified the local factor 
and immediately afterward, we saw that the dimension reduction allows us, 
when $k$ hyperbolic planes have been taken away and $B^m = B^{n-2k}$ is
anisotropic, to draw a connection between $\Pi^n$ and $\Pi^m$.  We conclude 
that, when $\be = 0$, the local factor is
\begin{equation*}
  \frac{\Pi^m(\al - n, k)}{q^{-\al} \, Z(\al)}
\end{equation*}
up to a multiplicative constant.

Finally, equations \eqref{e:g-per-odd} and \eqref{e:g-per-even} give us the
(uncorrected) global period.  Recall that when $\De = 1$ we use $\chi = \chi_0$ 
and the local factor of $L(\cdot, \chi) = \ze(\cdot)$ is $Z(\cdot)$, while
when $\De = -1$ we use $\chi = \chi_1$ and the local factor of $L(\cdot, \chi)$
is $1$.  Table~\ref{t:per} summarizes the conclusions.

\begin{table}
\caption{For each $n$, we list the uncorrected global period as well as the
correction factor at $p = 2$, ignoring multiplicative constants independent of 
$\al$.}\label{t:per}

\begin{tabular}{ccc}\toprule
$n$          & uncorrected period & correction factor at $p = 2$
\\\midrule
$3 + 8\ell$  & $\dfrac{\ze(\al - n)}{\ze(\al - 4\ell -1)}$
& $\dfrac{Z(\al - 4\ell - 1)}{q^{-\al}}$;
\\\addlinespace
$4 + 8\ell$  & $\dfrac{\ze(\al - n)\, L(\al - \frac n2, \chi)}{\ze(2\al - n)}$
& $\dfrac{Z(\al - \frac n2)}{q^{-\al}}$;
\\\addlinespace
$5 + 8\ell$  & $\dfrac{\ze(\al - n)}{L(\al - 4\ell -2, \chi)}$
& $\dfrac{Z(\al - 4\ell - 3)}{Z(2\al - n - 1) \, q^{-\al}}$;
\\\addlinespace
$6 + 8\ell$  & $\dfrac{\ze(\al - n)\, \ze(\al - \frac n2)}{\ze(2\al - n)}$
& $\dfrac{Z(2\al - n)}{Z(\al)\, Z(\al - n) \, Z(\al - \frac n2) \, q^{-\al}}$;
\\\addlinespace
$7 + 8\ell$  & $\dfrac{\ze(\al - n)}{\ze(\al - 4\ell -1)}$
& $\dfrac{Z(\al - 4\ell - 1)}
         {Z(\al - n - 1 - \log_q(1+q^{-4\ell-4}+q^{-n})) \, q^{-\al}}$;
\\\addlinespace
$8 + 8\ell$  & $\dfrac{\ze(\al - n)\, L(\al - \frac n2, \chi)}{\ze(2\al - n)}$
& $\dfrac{Z(2\al - n)}{Z(\al - \frac n2) \, q^{-\al}}$;
\\\addlinespace
$9 + 8\ell$  & $\dfrac{\ze(\al - n)}{L(\al - 4\ell -2, \chi)}$
& $\dfrac{1}{Z(\al - 4\ell - 5) \, q^{-\al}}$;
\\\addlinespace
$10 + 8\ell$ & $\dfrac{\ze(\al - n)\, \ze(\al - \frac n2)}{\ze(2\al - n)}$
& $\dfrac{1}
         {Z(\al - 4\ell - 6) \, q^{-\al}}$.
\\\addlinespace\bottomrule
\end{tabular}
\end{table}

\nocite{Ca,Ga,KaMuSu,MuSu,Om,PlRa,Wei1}


\begin{bibdiv}


\begin{biblist}

\bib{AiGoSa}{article}{
  author={Aizenbud, Avraham},
  author={Gourevitch, Dmitry},
  author={Sayag, Eitan},
  title={{$(O(V\oplus F),O(V))$} is a {G}elfand pair for any quadratic space {$V$} over a local field {$F$}},
  date={2009},
  journal={Math. Z.},
  volume={261},
  number={2},
  pages={239\ndash 244},
  review={\MR{2457297 (2010a:22020)}, \Zbl{1179.22017}},
}

\bib{Bo}{article}{
  author={Boavida, Jo{\~a}o Pedro},
  title={Compact periods of {E}isenstein series of orthogonal groups of rank one},
  date={2013},
  journal={Indiana U. Math. J.},
  volume={62},
  number={3},
  pages={869\ndash 890},
  review={\MR{3164848}, \Zbl{1301.11049}},
}

\bib{Bo2}{article}{
  author={Boavida, Jo{\~a}o Pedro},
  title={A spectral identity for second moments of {E}isenstein series of $\gO(n,1)$},
  date={2013},
  journal={Illinois J. Math.},
  volume={57},
  number={4},
  pages={1111\ndash1130},
  review={\MR{3285869}, \Zbl{1302.11031}},
}

\bib{Ca}{book}{
  author={Cassels, J. W.~S.},
  title={Rational quadratic forms},
  series={London Mathematical Society Monographs},
  publisher={Academic Press Inc.},
  address={London},
  date={1978},
  volume={13},
  review={\MR{522835 (80m:10019)}, \Zbl{0395.10029}},
}

\bib{DiGa1}{article}{
  label={DG09a},
  author={Diaconu, Adrian},
  author={Garrett, Paul},
  title={Integral moments of automorphic {$L$}-functions},
  date={2009},
  journal={J. Inst. Math. Jussieu},
  volume={8},
  number={2},
  pages={335\ndash 382},
  review={\MR{2485795 (2010h:11081)}, \Zbl{1268.11065}},
}

\bib{DiGa2}{article}{
  label={DG09b},
  author={Diaconu, Adrian},
  author={Garrett, Paul},
  title={Subconvexity bounds for automorphic {$L$}-functions},
  journal={J. Inst. Math. Jussieu},
  volume={9},
  date={2010},
  number={1},
  pages={95\ndash 124},
  review={\MR{2576799 (2011b:11064)}, \Zbl{1275.11084}},
}

\bib{DiGa3}{article}{
  label={DG09c},
  author={Diaconu, Adrian},
  author={Garrett, Paul},
  title={Averages of symmetric square {$L$}-functions, and applications},
  date={2009},
  status={preprint},
  eprint={http://www.math.umn.edu/~garrett/m/v/sym_two.pdf},
}

\bib{DiGaGo}{article}{
  author={Diaconu, Adrian},
  author={Garrett, Paul},
  author={Goldfeld, Dorian},
  title={Moments for {$L$}-functions for {$GL_r\times GL_{r-1}$}},
  conference={
    title={Patterson 60++ {I}nternational {C}onference on the {O}ccasion of the 60th {B}irthday of {S}amuel {J}.~{P}atterson},
    address={University of {G}{\"o}ttingen},
    date={July 27\ndash 29, 2009}, },
  book={
    title={Contributions in analytic and algebraic number theory},
    editor={Blomer, Valentin},
    editor={Mih{\u a}ilescu, Preda},
    series={Springer Proc. Math.},
    volume={9},
    publisher={Springer},
    place={New York}, },
  date={2012},
  pages={197\ndash 227},
  review={\MR{3060461}, \Zbl{06221503}},
}

\bib{Ga}{article}{
  author={Garrett, Paul},
  title={Euler factorization of global integrals},
  pages={35\ndash 101},
  book={
    title={Automorphic forms, automorphic representations, and arithmetic},
    editor={Doran, Robert~S.},
    editor={Dou, Ze-Li},
    editor={Gilbert, George~T.},
    series={Proc. Sympos. Pure Math.},
    publisher={Amer. Math. Soc.},
    address={Providence, R.I.},
    volume={66.2}, },
  date={1999},
  conference={
    title={NSF\ndash CBMS Regional Conference in Mathematics on Euler Products and Eisenstein Series},
    address={Texas Christian University, Fort Worth, Tex.},
    date={May 20\ndash 24, 1996}, },
  review={\MR{1703758 (2000m:11043)}, \Zbl{1002.11042}},
}

\bib{GrPr1}{article}{
  author={Gross, Benedict~H.},
  author={Prasad, Dipendra},
  title={Test vectors for linear forms},
  date={1991},
  journal={Math. Ann.},
  volume={291},
  number={2},
  pages={343\ndash 355},
  review={\MR{1129372 (92k:22028)}, \Zbl{0768.22004}},
}

\bib{GrPr2}{article}{
  author={Gross, Benedict~H.},
  author={Prasad, Dipendra},
  title={On the decomposition of a representation of {${\rm SO}\sb n$} when restricted to {${\rm SO}\sb {n-1}$}},
  date={1992},
  journal={Canad. J. Math.},
  volume={44},
  number={5},
  pages={974\ndash 1002},
  review={\MR{1186476 (93j:22031)}, \Zbl{0787.22018}},
}

\bib{GrPr3}{article}{
  author={Gross, Benedict~H.},
  author={Prasad, Dipendra},
  title={On irreducible representations of {${\rm SO}\sb {2n+1} \times {\rm SO}\sb {2m}$}},
  date={1994},
  journal={Canad. J. Math.},
  volume={46},
  number={5},
  pages={930\ndash 950},
  review={\MR{1295124 (96c:22028)}, \Zbl{0829.22031}},
}

\bib{GrRe}{article}{
  author={Gross, Benedict~H.},
  author={Reeder, Mark},
  title={From {L}aplace to {L}anglands via representations of orthogonal groups},
  date={2006},
  journal={Bull. Amer. Math. Soc. (N.S.)},
  volume={43},
  number={2},
  pages={163\ndash 205},
  review={\MR{2216109 (2007a:11159)}, \Zbl{1159.11047}},
}

\bib{IcIk}{article}{
  author={Ichino, Atsushi},
  author={Ikeda, Tamotsu},
  title={On the periods of automorphic forms on special orthogonal groups and the {G}ross--{P}rasad conjecture},
  date={2010},
  journal={Geom. Funct. Anal.},
  volume={19},
  number={5},
  pages={1378\ndash1425},
  review={\MR{2585578}, \Zbl{1216.11057}},
}

\bib{IwSa}{article}{
  author={Iwaniec, Henryk},
  author={Sarnak, Peter},
  title={Perspectives on the analytic theory of {$L$}-functions},
  date={2000},
  book={
    title={GAFA 2000},
    editor={Alon, N.},
    editor={Bourgain, J.},
    editor={Connes, A.},
    editor={Gromov, M.},
    editor={Milman, V.},
    publisher={Birk\"auser},
    address={Basel},
  },
  conference={
    title={Visions in Mathematics, towards 2000},
    address={Tel Aviv University},
    date={August 25\ndash September 3, 1999},
  },
  pages={705\ndash 741},
  note={\emph{Geom. Funct. Anal.}, Special Volume, Part II},
  review={\MR{1826269 (2002b:11117)}, \Zbl{0996.11036}},
}

\bib{KaMuSu}{article}{
   author={Kato, Shin-ichi},
   author={Murase, Atsushi},
   author={Sugano, Takashi},
   title={Whittaker--Shintani functions for orthogonal groups},
   journal={Tohoku Math. J. (2)},
   volume={55},
   date={2003},
   number={1},
   pages={1\ndash64},
   review={\MR{1956080 (2003m:22020)}, \Zbl{1037.22034}},
}

\bib{JaLaRo}{article}{
  author={Jacquet, Herv{\'e}},
  author={Lapid, Erez},
  author={Rogawski, Jonathan},
  title={Periods of automorphic forms},
  date={1999},
  journal={J. Amer. Math. Soc.},
  volume={12},
  number={1},
  pages={173\ndash 240},
  review={\MR{1625060 (99c:11056)}, \Zbl{1012.11044}},
}

\bib{Ji}{article}{
  author={Jiang, Dihua},
  title={Periods of automorphic forms},
  date={2007},
  book={
    title={Proceedings of the International Conference on Complex Geometry and Related Fields},
    editor={Yau, Stephen S.-T.},
    editor={Chen, Zhijie},
    editor={Wang, Jianpan},
    editor={Ten, Sheng-Li},
    series={AMS/IP Stud. Adv. Math.},
    volume={39},
    publisher={Amer. Math. Soc.},
    address={Providence, R.I.}, },
  pages={125\ndash 148},
  review={\MR{2338623 (2008k:11052)}, \Zbl{1161.11009}},
}

\bib{LaOf}{article}{
  author={Lapid, Erez},
  author={Offen, Omer},
  title={Compact unitary periods},
  date={2007},
  journal={Compos. Math.},
  volume={143},
  number={2},
  pages={323\ndash 338},
  review={\MR{2309989 (2008g:11091)}, \Zbl{1228.11073}},
}

\bib{LaRo}{article}{
  author={Lapid, Erez},
  author={Rogawski, Jonathan},
  title={Periods of {E}isenstein series},
  date={2001},
  journal={C. R. Acad. Sci. Paris S\'er. I Math.},
  volume={333},
  number={6},
  pages={513\ndash 516},
  review={\MR{1860921 (2002k:11072)}, \Zbl{1067.11028}},
}

\bib{Le}{article}{
  author={Letang, Delia},
  title={Automorphic spectral identities and applications to automorphic $L$-functions on $GL_2$},
  journal={J. Number Theory},
  volume={133},
  date={2013},
  number={1},
  pages={278\ndash317},
  review={\MR{2981412}, \Zbl{06110193}},
}

\bib{MuSu}{article}{
   author={Murase, Atsushi},
   author={Sugano, Takashi},
   title={Shintani function and its application to automorphic $L$-functions
   for classical groups. I. The case of orthogonal groups},
   journal={Math. Ann.},
   volume={299},
   date={1994},
   number={1},
   pages={17\ndash56},
   review={\MR{1273075 (96c:11054)}, \Zbl{0813.11032}},
}

\bib{Om}{book}{
  author={O'Meara, O.~Timothy},
  title={Introduction to quadratic forms},
  series={Classics in Mathematics},
  publisher={Springer-Verlag},
  address={Berlin},
  date={2000},
  note={Reprint of the 1973 edition},
  review={\MR{1754311 (2000m:11032)}, \Zbl{1034.11003}},
}

\bib{PlRa}{book}{
  author={Platonov, Vladimir},
  author={Rapinchuk, Andrei},
  title={Algebraic groups and number theory},
  series={Pure and Applied Mathematics},
  publisher={Academic Press Inc.},
  address={Boston, Mass.},
  date={1994},
  volume={139},
  note={Translated from the 1991 Russian original by Rachel Rowen},
  review={\MR{1278263 (95b:11039)}, \Zbl{0841.20046}},
}

\bib{Sak}{article}{
  author={Sakellaridis, Yiannis},
  title={Spherical varieties and integral representations of $L$-functions},
  journal={Algebra Number Theory},
  volume={6},
  date={2012},
  number={4},
  pages={611\ndash667},
  review={\MR{2966713}, \Zbl{1253.11059}},
}

\bib{SaVe}{article}{
  author={Sakellaridis, Yiannis},
  author={Venkatesh, Akshay},
  title={Periods and harmonic analysis on spherical varieties},
  date={2012},
  status={preprint},
  eprint={arXiv:1203.0039v1 [math.RT]},
}

\bib{Wei1}{book}{
  author={Weil, Andr{\'e}},
  title={Adeles and algebraic groups},
  series={Progress in Mathematics},
  publisher={Birkh\"auser},
  address={Boston, Mass.},
  date={1982},
  volume={23},
  review={\MR{670072 (83m:10032)}, \Zbl{0493.14028}},
}


\end{biblist}

\end{bibdiv}
\end{document}